\documentclass[11pt]{article}
\usepackage{amssymb}
\usepackage{graphicx}
\usepackage{amsmath}
\usepackage{amsthm}
\usepackage{enumitem}
\usepackage{amsfonts}
\usepackage{authblk}
\usepackage{amssymb}
\usepackage{mathabx}
\usepackage{amsmath,amssymb}
\usepackage{xcolor}
\newtheorem{theorem}{Theorem}
\newtheorem{remark}{Remark}%

\newtheorem{corollary}{Corollary}%
\title{Polynomial approximation of derivatives by the constrained mock-Chebyshev least squares operator}
\date{}
\author[1]{F. Dell'Accio}
\author[1]{ F. Nudo }
\affil[1]{Department of Mathematics and Computer Science, University of Calabria, Rende (CS), Italy}
\begin{document}
\maketitle

\begin{abstract}
The constrained mock-Chebyshev least squares operator is a linear approximation operator based on an equispaced grid of points. Like other polynomial or rational approximation methods, it was recently introduced in order to defeat the Runge phenomenon that occurs when using polynomial interpolation on large sets of equally spaced points. The idea is to improve the mock-Chebyshev subset interpolation, where the considered function $f$ is interpolated only on a proper subset of the uniform grid, formed by nodes that mimic the behavior of Chebyshev--Lobatto nodes. In the mock-Chebyshev subset interpolation all remaining nodes are discarded, while in the constrained mock-Chebyshev least squares interpolation they are used in a simultaneous regression, with the aim to further improving the accuracy of the approximation provided by the mock-Chebyshev subset interpolation. The goal of this paper is two-fold. We discuss some theoretical aspects of the constrained mock-Chebyshev least squares operator and present new results. In particular, we introduce explicit representations of the error and its derivatives. Moreover, for a sufficiently smooth function $f$ in $[-1,1]$, we present a method for approximating the successive derivatives of $f$ at a point $x\in [-1,1]$, based on the constrained mock-Chebyshev least squares operator and provide estimates for these approximations.
Numerical tests demonstrate the effectiveness of the proposed method.  
\end{abstract}

\section{Introduction}

 Let $X_n=\{x_0,\dots,x_n\}$ be the set of $n+1$ equispaced nodes in $[-1,1]$, that is 
$$x_i=-1+\frac{2}{n}i, \quad i=0,\dots,n. $$
In order to define the constrained mock-Chebyshev least squares linear operator, some settings are needed. 
 We set  $$m=\left\lfloor \pi \sqrt{\frac{n}{2}} \right\rfloor, \quad  p=\left\lfloor\frac{\pi}{\sqrt{2}}\sqrt{\frac{{{n}}}{{{6}}}} \right\rfloor,$$
and we consider a basis  $\mathcal{B}=\{u_0(x),\dots,u_r(x)\}$ of the polynomial space $\Pi_r$, constituted by polynomials of degree less than or equal to $r=m+p+1$. We denote by 
$X^{\prime}_m=\{x^{\prime}_k\}_{k=0}^m$ the subset of mock-Chebyshev nodes~\cite{Boyd:2009:DFL, Ibrahimoglu:2020:AFA, Ibrahimoglu:2021:ANA}, that is the nodes of the uniform grid $X_n$ which best mimic the behavior of the well known Chebyshev--Lobatto nodes and by $X^{\prime\prime}_{n-m}=X_n\setminus X^{\prime}_m$ the relative complement of $X^{\prime}_m$ with respect to $X_n$. In~\cite{DeMarchi:2015:OTC} it is proven that,
when $n$ is sufficiently large, we can approximate an equispaced grid of $q=\left\lfloor\frac{n}{6}\right\rfloor$ internal nodes of $[-1,1]$ with nodes which belong to $X^{\prime\prime}_{n-m}$. We denote this grid by $\tilde{X}^{\prime\prime}_{n-m},$ and by $X^{\prime\prime\prime}_p=\{x^{\prime\prime\prime}_{k}\}_{k=0}^p$ the mock-Chebyshev subset of $\tilde{X}^{\prime\prime}_{n-m}$. We suppose to have reordered the set $X_n$ so that its first $m+1$ points are those of $X^{\prime}_m$. We suppose also that the first $m+1$ elements of the basis $\mathcal{B}$ span the polynomial space $\Pi_m$.  We define 
the constrained mock-Chebyshev least squares linear operator as follows 
\begin{eqnarray*}
\hat{P}_{r,n}:  C([-1,1]) &\rightarrow& C([-1,1])
\\
f(x) &\mapsto& \hat{P}_{r,n}[f](x)=\sum_{i=0}^r\hat{a}_iu_i(x), \qquad x\in[-1,1],
\end{eqnarray*}
where $\hat{\boldsymbol{a}}=[\hat{a}_0,\hat{a}_1,\dots,\hat{a}_r]^T$ is the solution of the KKT linear equations~\cite{Boyd:2018:ITA}
\begin{equation}
\label{Cons_Least_Squares_Prob}
    \begin{bmatrix}
2 V^TV & C^T  \\
C & 0  \\
\end{bmatrix}
\begin{bmatrix}
\hat{\boldsymbol{a}} \\
\hat{\boldsymbol{z}} \\
\end{bmatrix}=
\begin{bmatrix}
2 V^T\boldsymbol{b} \\
\boldsymbol{d} \\
\end{bmatrix}.
\end{equation}
In the equation~\eqref{Cons_Least_Squares_Prob} the matrices $V$ and $C$ are defined as follows
\begin{equation*}
   V=[u_j(x_i)]_{\substack{i=0,\dots,n\\ j=0,\dots,r}},  \qquad C=[u_j(x_i)]_{\substack{i=0,\dots,m\\ j=0,\dots,r}},
\end{equation*}
 $\boldsymbol{b}=[f(x_0),\dots,f(x_n)]^T$,  $\boldsymbol{d}=[f(x_0),\dots,f(x_m)]^T$, 
and $\hat{\boldsymbol{z}}=[\hat{z_1},\dots,\hat{z}_{m+1}]^T$ is the Lagrange multipliers vector. The matrix 
\begin{equation}
\label{KKTMatrix}
M=
    \begin{bmatrix}
2 V^TV & C^T  \\
C & 0  \\
\end{bmatrix}
\end{equation}
 is non singular~\cite{DellAccio:2022:GOT} and is called KKT matrix, in honor to W. Karush, H.W. Kuhn and A. Tucker~\cite[Chapter 16]{Boyd:2018:ITA}. The operator $\hat{P}_{r,n}$ satisfies the following properties:
\begin{itemize}
     \item[$\textbf{\textit{i}}$)]  $\hat{P}_{r,n}$ is a linear operator~\cite{DellAccio:2022CMC}, that is
    \begin{equation}
        \label{eq:PropLinearity}
        \hat{P}_{r,n}\left[\lambda f+\mu g\right]=\lambda\hat{P}_{r,n}[f] +\mu \hat{P}_{r,n}[g], \quad  f,g\in C([-1,1]), \, \quad \lambda,\mu\in\mathbb{R};
    \end{equation}
    \item[$\textbf{\textit{ii}}$)] the range of $\hat{P}_{r,n}$ is $\Pi_r([-1,1])$;
    \item[$\textbf{\textit{iii}}$)] $\hat{P}_{r,n}$ reproduces polynomials of degree $\le r$~\cite{DeMarchi:2015:OTC}, that is
    \begin{equation}
        \label{eq:PropReprodPolyn}
        \hat{P}_{r,n}[q]=q, \quad \text{ for each } q\in\Pi_r;
    \end{equation}
\item[$\textbf{\textit{iv}}$)] $\hat{P}_{r,n}$ is idempotent, that is
    $$\hat{P}_{r,n}^2=\hat{P}_{r,n};$$
    \item[$\textbf{\textit{v}}$)]
       $\hat{P}_{r,n}[f]$ is completely determined by the evaluations of $f$ on the grid $X_n$, in particular 
    \begin{equation}
    \label{eq:PropNoInj}
     \hat{P}_{r,n}[f]=\hat{P}_{r,n}\left[P_n[f]\right], \quad \text{ for each } \, f\in C([-1,1]),
\end{equation}
where $$P_n[f](x)=\sum_{i=0}^nf(x_i)\ell_i(x), \quad x\in[-1,1],$$ 
and 
$$\ell_i(x)=\prod\limits_{\substack{j=0\\ j\neq i}}^n \frac{x-x_j}{x_i-x_j}, \quad i=0,\dots, n,  \quad x\in[-1,1].$$ 
\item[$\textbf{\textit{vi}}$)] $\hat{P}_{r,n}$ interpolates the function $f$ at the mock-Chebyshev subset of nodes, that is 
\begin{equation*}
    \hat{P}_{r,n}[f](x^{\prime}_i)=f(x^{\prime}_i), \quad i=0,\dots,m.
\end{equation*}
\end{itemize}
The constrained mock-Chebyshev least squares operator has been very recently generalized to the bivariate case~\cite{DellAccio:2022:GOT} and used to introduce accurate and stable quadrature formulas on equispaced nodes~\cite{DellAccio:2022CMC}. Despite some bound in uniform norm of the error of approximation 
\begin{equation}
    \label{eq:remainder}
    \hat{R}_{r,n}[f](x):= f(x)-\hat{P}_{r,n}[f](x), \quad x\in[-1,1],
\end{equation}
was already discussed in~\cite{DeMarchi:2015:OTC}, any explicit representation of the error $\hat{R}_{r,n}[f]$ has not been given yet. In this paper we face this problem and provide a pointwise representation of the error which takes into account the peculiarity of $\hat{P}_{r,n}[f]$ of being a mixed interpolation-regression polynomial. Starting from this representation, bounds of the error~\eqref{eq:remainder} in uniform norm will be given through the operator norm of $\hat{P}_{r,n}$ and the error of best uniform approximation by polynomials of degree less than or equal to $r$. For a sufficiently smooth function $f$ in $[-1,1]$, we obtain explicit representations of the derivative of the error~\eqref{eq:remainder} by differentiation and bounds in uniform norm by applying the Markov Theorem~\cite{Markoff:1916:UPD}. As an application, we introduce  a new differentiation method for approximating the successive derivatives of a sufficiently smooth function $f$ in $[-1,1]$ based on the constrained mock-Chebyshev least squares operator. 
Furthermore, we prove an iterative relationship between the coefficients of $\hat{P}^{(\nu)}_{r,n}$ and those of  $\hat{P}^{(\nu-1)}_{r,n}$, when they are expressed in the Chebyshev polynomial basis of first kind. 

The paper is organized as follows. In Section~\ref{Section:TheoreticallyAspects}  we discuss some theoretical aspects of the constrained mock-Chebyshev least squares operator. We present a theoretical bound for the norm of this operator and we introduce explicit representations of the error and its derivatives. Basing on this operator, in Section~\ref{Sec:NumericalDifferentiation }, we present a differentiation method for approximating the successive derivatives of a sufficiently smooth function $f$ at any point $x\in[-1,1]$. This method provides 
 a global polynomial approximations of the successive derivatives of $f$. 
The accuracy of this method is proved by several numerical examples, presented in Section~\ref{Sec:NumericalExp}.

\section{Constrained mock-Chebyshev least squares linear operator: theoretical aspects}
\label{Section:TheoreticallyAspects}

In this Section, we provide pointwise representations of the error~\eqref{eq:remainder} in terms of finite differences of the function $f\in C([-1,1])$ or its appropriate derivative, if $f$ is sufficiently smooth, and give related bounds in uniform norm. If $f$ is sufficiently smooth in $[-1,1]$ we provide explicit representations for the successive derivatives of the error~\eqref{eq:remainder} and give related bounds in uniform norm. To this aim, some preliminary results are needed. First of all, we give a bound for the norm of the operator $\hat{P}_{r,n}$,
\begin{equation}
\label{eq:NormOfOperator}
\left\lVert \hat{P}_{r,n} \right\rVert:= \sup\limits_{\substack{f\in C([-1,1])\\ \left\lVert f \right\rVert_{\infty}\le1}} \left\lVert \hat{P}_{r,n}[f] \right\rVert_{\infty}.
\end{equation} 

\begin{theorem}
\label{Thm:EstimateOfNorm}
The constrained mock-Chebyshev least squares operator is bounded. In particular 
\begin{equation}
\label{eq:BoundThm}
    \left\lVert \hat{P}_{r,n} \right\rVert \le C\left(2 (r+1)\kappa(M)+(m+1)\left\lVert
 M^{-1}\right\rVert_1\right),
\end{equation}
where 
\begin{equation}
    \label{eq:ConstantC}
   C:=\max_{j=0,\dots,r}\left\lVert u_j \right\rVert_{\infty}, \qquad \kappa(M)=\left\lVert
 M
\right\rVert_1 \lVert
 M^{-1}
\rVert_1.
\end{equation}
\end{theorem}
\begin{proof}
Let $f\in C([-1,1])$ be a continuous function such that $\left\lVert f \right\rVert_{\infty}\le 1$. By using the triangular inequality we get
\begin{equation}
\label{eq:ContinuityCMCLSO}
    \left\lvert \hat{P}_{r,n}[f](x) \right\rvert=\left\lvert \sum_{i=0}^r \hat{a}_iu_i(x) \right\rvert\le \sum_{i=0}^r \left\lvert\hat{a}_i\right\rvert \left\lvert u_i(x)\right\rvert, \quad x\in[-1,1].
\end{equation}
From~\eqref{eq:ContinuityCMCLSO}, by passing to the supremum with respect to $x\in[-1,1]$, by the setting~\eqref{eq:ConstantC} we get
\begin{equation} 
\label{eq:BoundP(f)}
   \left \lVert \hat{P}_{r,n}[f] \right\rVert_{\infty} \le C \sum_{i=0}^r \left\lvert\hat{a}_i\right\rvert=C\left\lVert\hat{\boldsymbol{a}}\right\rVert_1.
\end{equation}
We bound the $1$-norm of the coefficients vector $\hat{\boldsymbol{a}}$ as follows.
Since the KKT matrix~\eqref{KKTMatrix} is not singular, from the linear system~\eqref{Cons_Least_Squares_Prob} we get
\begin{eqnarray*}
\left\lVert
\hat{\boldsymbol{a}} 
\right\rVert_1
\le
\left\lVert
\begin{bmatrix}
\hat{\boldsymbol{a}} \\
\hat{\boldsymbol{z}} \\
\end{bmatrix}
\right\rVert_1&\le&
\left\lVert
 M^{-1}
\right\rVert_1
\left\lVert
\begin{bmatrix}
2 V^T\boldsymbol{b} \\
\boldsymbol{d} \\
\end{bmatrix}
\right\Vert_1\\
&=& \left\lVert
 M^{-1}
\right\rVert_1\left(2 \sum_{j=0}^r \left\lvert\sum_{i=0}^n u_j(x_i)f(x_i)\right\rvert +\sum_{j=0}^m \left\lvert f(x_j)\right\rvert\right) \\
&\le& \left\lVert
 M^{-1}
\right\rVert_1\left(2 \sum_{j=0}^r \sum_{i=0}^{n} \left\lvert u_j(x_i)\right\rvert+m+1\right)\\
&\le& \left\lVert
 M^{-1}
\right\rVert_1\left(2 \left\lVert M\right\rVert_1(r+1)+m+1\right).
\end{eqnarray*}
Therefore
\begin{equation}
    \label{eq:a1}
\left\lVert
\hat{\boldsymbol{a}} 
\right\rVert_1    \le 2 (r+1)\kappa(M)+(m+1)\left\lVert
 M^{-1}
\right\rVert_1.
\end{equation}
Finally, by using the bound~\eqref{eq:a1} in~\eqref{eq:BoundP(f)}, we get
\begin{equation}
\label{eq:polval}
   \left \lVert \hat{P}_{r,n}[f] \right\rVert_{\infty} \le C\left(2 (r+1)\kappa(M)+(m+1)\left\lVert
 M^{-1}\right\rVert_1\right).
\end{equation}
Since the right-hand side of~\eqref{eq:polval} does not depend on $f$, from~\eqref{eq:NormOfOperator}, we have
\begin{equation*}
\left \lVert \hat{P}_{r,n} \right\rVert\le C\left(2 (r+1)\kappa(M)+(m+1)\left\lVert
 M^{-1}\right\rVert_1\right).
\end{equation*}
\end{proof}
\begin{table}
\begin{center}
\label{Tab: normM-1}%
\begin{tabular}{llllllll}
$n$ & 100  & 500 & 1000 & 5000 & 10000 & 50000 & 100000\\
\hline
 $\kappa(M)$ & 7.79e+03 & 8.45e+04  & 2.43e+05 & 2.86e+06 &  8.34e+06 &   1.00e+08 &  2.91e+08 \\
 $\left\lVert
 M^{-1}
\right\rVert_1$ & 21.80 & 52.90  & 78.20 & 186.75 &  275.55 &   661.36 &  965.75 \\
\end{tabular}
\caption{Values of the condition number $\kappa(M)$ and of the norm $\left\lVert
 M^{-1}
\right\rVert_1$ in correspondence of different values of $n$ ranging from $n=100$ to $n=100000$, by using the Chebyshev polynomial basis of the first kind.}
\end{center}
\end{table}
\begin{remark}
\label{remark_Norm1ChebPolBasis}
In the case of the Chebyshev polynomial basis of the first kind
$$\mathcal{B}_{C,1}=\{T_0(x),\dots,T_r(x)\}, \quad T_k(x)=\cos({k\arccos{x}}), \quad x\in[-1,1], \quad k=0,\dots,r,
$$
we have $\left\lVert T_k \right\rVert_{\infty}=1$ for each $k=0,\dots,r$~\cite[Chapter 1]{Rivlin:1975:PSL}.
In this case equation~\eqref{eq:BoundThm} becomes
\begin{equation}
\label{eq:newBound}
    \left\lVert \hat{P}_{r,n} \right\rVert_{\infty} \le 2 (r+1)\kappa(M)+(m+1)\left\lVert
 M^{-1}\right\rVert_1.
\end{equation}
\end{remark}

In order to appreciate the quality of the bound~\eqref{eq:newBound}, we explicitly compute $\kappa(M)$ and $\left\lVert M^{-1}\right\rVert_1$ for different values of $n$ ranging from $100$ to $100000$. These values are shown in Table~\ref{Tab: normM-1}. 
 We computed also the value 
\begin{equation}
\label{bound1}
 B_n:= 2 (r+1)\kappa(M)+(m+1)\left\lVert
 M^{-1}\right\rVert_1 
\end{equation}
 of the bound~\eqref{eq:newBound} of the norm $\left\lVert \hat{P}_{r,n}\right\rVert$ relative to sets of $n+1$ equispaced nodes with $n$ ranging from $100$ to $100000$. The results are represented in Figure~\ref{Fig:normoperator}, where the plot is realized on a log-log scale. The plot shows a linear relation between the logarithm of $n$ and the logarithm of the bound $B_n$. We computed the coefficients of this relation through a linear regression and after standard computations, we found
$$B_n\approx e^{3.66}n^{2.03}.$$
Figure~\ref{Fig:normoperator} contains the approximations of the bound $B_n$ computed through the regression line, as well. 

\begin{figure}[h]
\begin{center}
    \includegraphics[width=.49\textwidth]{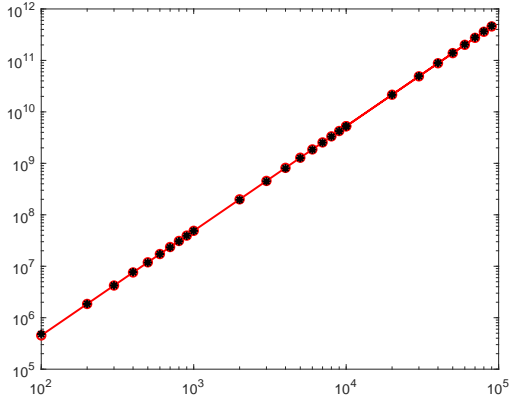} 
         \caption{Log-log plot of the values $B(n)$ (in black stars, \textcolor{black}{$\star$}) of the bound~\eqref{eq:newBound} of the norm $\left\lVert \hat{P}_{r,n}\right\rVert$  relative to sets of $n+1$ equispaced nodes, with $n$ ranging from $10^2$ to $10^5$. The linear relation between the values on the $x$-axis and those on the $y$-axis is evident and confirmed by the closeness of their approximations (in red circles, \textcolor{red}{o}) computed through a linear regression. }
    \label{Fig:normoperator}
    \end{center}
\end{figure}

In view of Properties $\textbf{\textit{i}})$-$\textbf{\textit{iv}})$ and Theorem~\ref{Thm:EstimateOfNorm}, the constrained mock-Chebyshev least squares operator is a projection on the polynomial space $\Pi_r$~\cite[Chapter 6]{Cheney:1998:ITA}. As a consequence, it is possible to give standard estimation for the approximation error
\begin{equation}
\label{eq:ApproximationError}
    E[f]:=\left\lVert f-\hat{P}_{r,n}[f]\right\rVert_{\infty}, \quad f\in C([-1,1]).
\end{equation}
With this aim, for computational convenience and to short the notation, we suppose to work with the Chebyshev polynomial basis of first kind $\mathcal{B}_{C,1}$. 
\begin{theorem}
\label{Thm:Errorb}
Let be $f\in C([-1,1])$, then 
\begin{equation*}
    E[f]\le \left(1+B_n\right)\left\lVert f-p_r^{\star}\right\rVert_{\infty},
\end{equation*}
where $p_r^{\star}$ is the polynomial of best uniform approximation of $f$ of degree less than or equal to $r$. 
\end{theorem}
\begin{proof}
By the Properties $\textbf{\textit{i}})$-$\textbf{\textit{iii}})$, we easily find
\begin{eqnarray*}
    \left\lVert f-\hat{P}_{r,n}[f]\right\rVert_{\infty}&=&    \left\lVert f-p_r^{\star}+p_r^{\star}-\hat{P}_{r,n}[f]\right\rVert_{\infty}\\
    &=&\left\lVert f-p_r^{\star}-\hat{P}_{r,n}[f-p_r^{\star}]\right\rVert_{\infty}\le \left(1+\left\lVert \hat{P}_{r,n}\right\rVert\right)\left\lVert f-p_r^{\star}\right\rVert_{\infty}.
\end{eqnarray*}
The result then follows from  Theorem~\ref{Thm:EstimateOfNorm} after setting~\eqref{bound1}.
\end{proof}
\begin{corollary}
\label{corollaryJackson}
Let be $f\in C^k([-1,1])$, $k=0,\dots,r$, then we have
\begin{equation}
\label{NormError1}
   E[f]\le \left(1+B_n\right) \omega_f\left(\frac{\pi}{r+1}\right), \quad k=0,
\end{equation}
\begin{equation}
\label{NormErrorEstimate}
   E[f]\le \left(\frac{\pi}{2}\right)^k\left(1+B_n\right) \frac{\left\lVert f^{(k)}\right\rVert_{\infty}}{(r+1)r\cdots (r-k+2)}, \quad 0<k\le r,
\end{equation}
where $\omega_f(\cdot)$ is the modulus of continuity of the function $f$~\cite{Cheney:1998:ITA}. 
\end{corollary}
\begin{proof}
The result follows from Theorem~\ref{Thm:Errorb} and Jackson Theorem~\cite[Chapter 4]{Cheney:1998:ITA}. 
\end{proof}

\begin{figure}
\begin{center}
    \includegraphics[width=.49\textwidth]{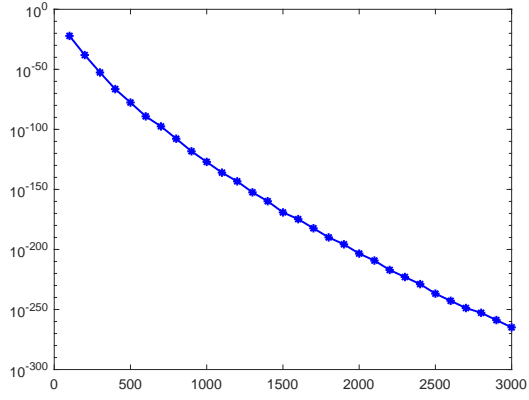} 
         \caption{Semilog plot of the values $\left(\frac{\pi}{2}\right)^r \frac{\left(1+B_n\right)}{(r+1)!}$ for $n=100k$, $k=1,\dots,30$.}
    \label{Fig:StimaErrorFun}
    \end{center}
\end{figure}

Let $P_r[f]$ be the Lagrange interpolation polynomial of the function $f$ at the node set $X^{\prime}_m\cup X^{\prime\prime\prime}_p$, that is
\begin{equation}
    P_r[f](x)=\sum_{i=0}^m \ell_{i,m}(x) f(x^{\prime}_i)+\sum_{j=0}^p \ell_{j,p}(x) f(x^{\prime\prime\prime}_j), \quad x\in[-1,1],
\end{equation}
where 
\begin{equation*}
    \ell_{i,m}(x)=\prod\limits_{\substack{k=0\\ k\neq i}}^{m}\frac{x-x^{\prime}_k}{x^{\prime}_i-x^{\prime}_k}\prod_{s=0}^{p}\frac{x-x^{\prime\prime\prime}_s}{x^{\prime}_i-x^{\prime\prime\prime}_s}, \qquad     \ell_{j,p}(x)=\prod_{k=0}^{m}\frac{x-x^{\prime}_k}{x^{\prime\prime\prime}_j-x^{\prime}_k}\prod\limits_{\substack{s=0\\ s\neq j}}^{p}\frac{x-x^{\prime\prime\prime}_s}{x^{\prime\prime\prime}_j-x^{\prime\prime\prime}_s}, \end{equation*}
and let \begin{equation}
    \label{error:LagrInter}
    R_{r}[f](x):=f(x)-P_r[f](x), \quad x\in[-1,1],
\end{equation}
the error of Lagrange interpolation. The following Theorem gives a pointwise representation of the error~\eqref{eq:remainder} for all $x\in[-1,1]$. 
\begin{theorem}
\label{Thm: EstimateDerivativeDiff}
Let be $f\in C([-1,1])$, then
\begin{equation}
    \hat{R}_{r,n}[f](x)=f(x)-\hat{P}_{r,n}[f](x)=R_r[f](x)+\sum_{j=0}^p \ell_{j,p}(x) \hat{R}_{r,n}[f](x^{\prime\prime\prime}_j).
\end{equation}
\end{theorem}
\begin{proof}
By Property $\textbf{\textit{ii}}$) $\hat{P}_{r,n}[f]\in\Pi_r$ and by Property $\textbf{\textit{vi}}$) $\hat{P}_{r,n}[f](x^{\prime}_i)=f(x^{\prime}_i)$, $i=0,\dots,m$. Then by the uniqueness of the Lagrange interpolation polynomial, we get 
   \begin{eqnarray*}
     &&\hat{P}_{r,n}[f](x)=P_r[  \hat{P}_{r,n}[f]](x)=\sum_{i=0}^m \ell_{i,m}(x) \hat{P}_{r,n}[f](x^{\prime}_i)+\sum_{j=0}^p \ell_{j,p}(x) \hat{P}_{r,n}[f](x^{\prime\prime\prime}_j)\\
     &=& \sum_{i=0}^m \ell_{i,m}(x) f(x^{\prime}_i)+\sum_{j=0}^p \ell_{j,p}(x) \hat{P}_{r,n}[f](x^{\prime\prime\prime}_j)\\
     &=& \sum_{i=0}^{m} \ell_{i,m}(x) f(x^{\prime}_i)+\sum_{j=0}^{p} \ell_{j,p}(x) f(x^{\prime\prime\prime}_j)+\sum_{j=0}^p \ell_{j,p}(x) \left(\hat{P}_{r,n}[f](x^{\prime\prime\prime}_j)-f(x^{\prime\prime\prime}_j)\right)\\
     &=& P_r[f](x)+\sum_{j=0}^p \ell_{j,p}(x) \left(\hat{P}_{r,n}[f](x^{\prime\prime\prime}_j)-f(x^{\prime\prime\prime}_j)\right).
   \end{eqnarray*}
Therefore
\begin{eqnarray*}
f(x)-\hat{P}_{r,n}[f](x)&=&f(x)-P_r[f](x)+\sum_{j=0}^p \ell_{j,p}(x) \left(f(x^{\prime\prime\prime}_j)-\hat{P}_{r,n}[f](x^{\prime\prime\prime}_j)\right)\\ &=&R_r[f](x)+\sum_{j=0}^p \ell_{j,p}(x) \hat{R}_{r,n}[f](x^{\prime\prime\prime}_j).
\end{eqnarray*}
\end{proof}
Let be $f\in C^{r+1}([-1,1])$. In this case, the Peano kernel Theorem~\cite{Davis:1975:IAA, Howell:1991:DEB} allows us to represent the remainder $R_r[f]$ of Lagrange interpolation on the node set $X^{\prime}_m\cup X^{\prime\prime\prime}_p$ in integral form
\begin{equation}
\label{RemainderPeanoKernel}
R_r[f](x)=\int_{-1}^1 K(x,t)\frac{f^{(r+1)}(t)}{r!}dt, \quad x\in[-1,1],
\end{equation}
where 
\begin{equation*}
K(x,t)=(x-t)^r_{+}-P_r[(x-t)^r_{+}], \quad x,t\in[-1,1],    
\end{equation*}
and 
\begin{equation*}
    (x-t)^r_{+}=\left\{\begin{array}{ll}
    (x-t)^r, & \text{if } x-t\ge 0,  \\
    0, & \text{if } x-t<0.
\end{array} \right.
\end{equation*}
 We can differentiate both members of~\eqref{RemainderPeanoKernel} $\nu$ times, $\nu=1,\dots,r$, with respect to $x$, in order to obtain pointwise representations of the successive derivatives of the error of Lagrange interpolation
 \begin{equation}
     \label{errorpointwisederiv}
     R_r^{(\nu)}[f](x)=\int_{-1}^1\frac{\partial^{\nu} K(x,t)}{\partial x^\nu}\frac{f^{(r+1)}(t)}{r!}dt.
 \end{equation}
By using these representations, pointwise and uniform bounds for $R_r^{(\nu)}[f]$ are obtained~\cite{Howell:1991:DEB}. 
In particular, in~\cite{Howell:1991:DEB} the following uniform bounds is proven 
\begin{equation}
\label{BoundDerivPolLag}
    \left\lVert R_r^{(\nu)}\right\rVert_{\infty}\le \left\lVert \omega_r^{(\nu)}\right \rVert_{\infty}\frac{\left\lVert f^{(r+1)}\right \rVert_{\infty}}{\nu!(r+1-\nu)!}, \quad \nu=1,\dots,r,
\end{equation}
where 
\begin{equation*}
    \omega_r(x)=\prod_{k=0}^m(x-x^{\prime}_k)\prod_{s=0}^p (x-x^{\prime\prime\prime}_s), \quad x\in[-1,1].
\end{equation*}
For the derivative of the remainder of the mock-Chebyshev least squares interpolation, the following bounds hold.
\begin{theorem}
Let be $f\in C^{r+1}([-1,1])$ and $\nu=1,\dots,r$, then
\begin{eqnarray}
\label{boundDerError}
 \left\lVert \hat{R}_{r,n}^{(\nu)}[f]\right\rVert_{\infty}&=&\left\lVert f^{(\nu)}-\hat{P}_{r,n}^{(\nu)}[f]\right\rVert_{\infty} \notag \\ &\le& \frac{\prod\limits_{j=0}^{\nu-1}\left((r+1)^2-j^2\right)}{\prod\limits_{j=0}^{\nu-1}(2j+1)}
 \left\lVert \omega_r\right \rVert_{\infty}\frac{\left\lVert f^{(r+1)}\right \rVert_{\infty}}{\nu!(r+1-\nu)!}\notag\\
 &&+\frac{\prod\limits_{j=0}^{\nu-1}\left(r^2-j^2\right)}{\prod\limits_{j=0}^{\nu-1}(2j+1)}\sum\limits_{k=0}^p\left\lVert \ell_{k,p} \right\rVert_{\infty}\left(\frac{\pi}{2}\right)^r\left(1+B_n\right) \frac{\left\lVert f^{(r)}\right\rVert_{\infty}}{(r+1)!}.
\end{eqnarray}
\end{theorem}
\begin{proof}
For each $\nu=1,\dots,r$, by Theorem~\ref{Thm: EstimateDerivativeDiff}, we have
\begin{equation*}
    \hat{R}_{r,n}^{(\nu)}[f](x)=R_r^{(\nu)}[f](x)+\sum_{j=0}^p \ell^{(\nu)}_{j,p}(x) \left(\hat{P}_{r,n}[f](x^{\prime\prime\prime}_j)-f(x^{\prime\prime\prime}_j)\right).
\end{equation*}
By using the triangular inequality 
\begin{equation*}
    \left\lVert \hat{R}_{r,n}^{(\nu)}[f]\right\rVert_{\infty}\le   \left\lVert R_{r}^{(\nu)}[f]\right\rVert_{\infty}+\sum_{j=0}^p \left\lVert\ell^{(\nu)}_{j,p}\right\rVert_{\infty} \left\lvert\hat{R}_{r,n}[f](x^{\prime\prime\prime}_j)\right\rvert.
\end{equation*}
The result follows from~\eqref{BoundDerivPolLag} by applying the Markov's inequality~\cite{Markoff:1916:UPD} to bound both $ \left\lVert R_{r}^{(\nu)}[f]\right\rVert_{\infty}$ and $\left\lVert\ell^{(\nu)}_{j,p}\right\rVert_{\infty}$, and the Corollary~\ref{corollaryJackson} to bound $\left\lvert\hat{R}_{r,n}[f](x^{\prime\prime\prime}_j)\right\rvert$.
\end{proof}

In Figures~\ref{Fig:Casen100}, \ref{Fig:Casen1000} and~\ref{Fig:Casen10000} the nodal polynomial $\omega_m$ at the mock-Chebyshev nodes $X^{\prime}_m$, the nodal polynomial $\omega_r$ at the node set $X^{\prime}_m\cup X^{\prime\prime\prime}_p$ and the Lagrange fundamental polynomials at the node set $X^{\prime}_m\cup X^{\prime\prime\prime}_p$ are graphically represented for $n=100,1000,10000$. It is worth noting that the uniform norm of $\omega_r$ is always less than the uniform norm of $\omega_m$ and the ratio $\left\lVert\omega_r\right\rVert_{\infty}/\left\lVert\omega_m\right\rVert_{\infty}$ decreases exponentially to zero as $n$ increases.
 \begin{figure}
 \begin{center}
          \includegraphics[width=.32\textwidth]{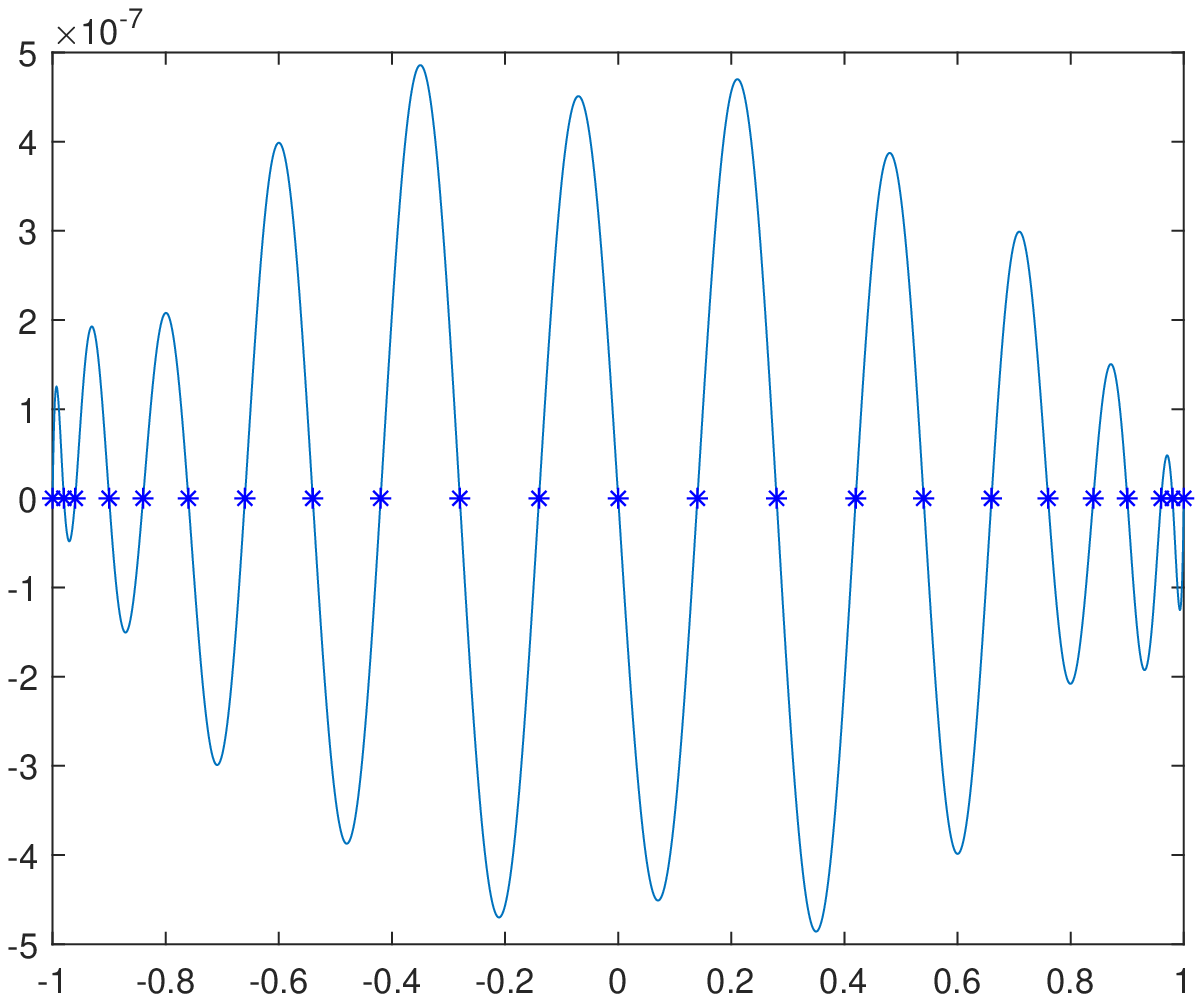}
    \includegraphics[width=.32\textwidth]{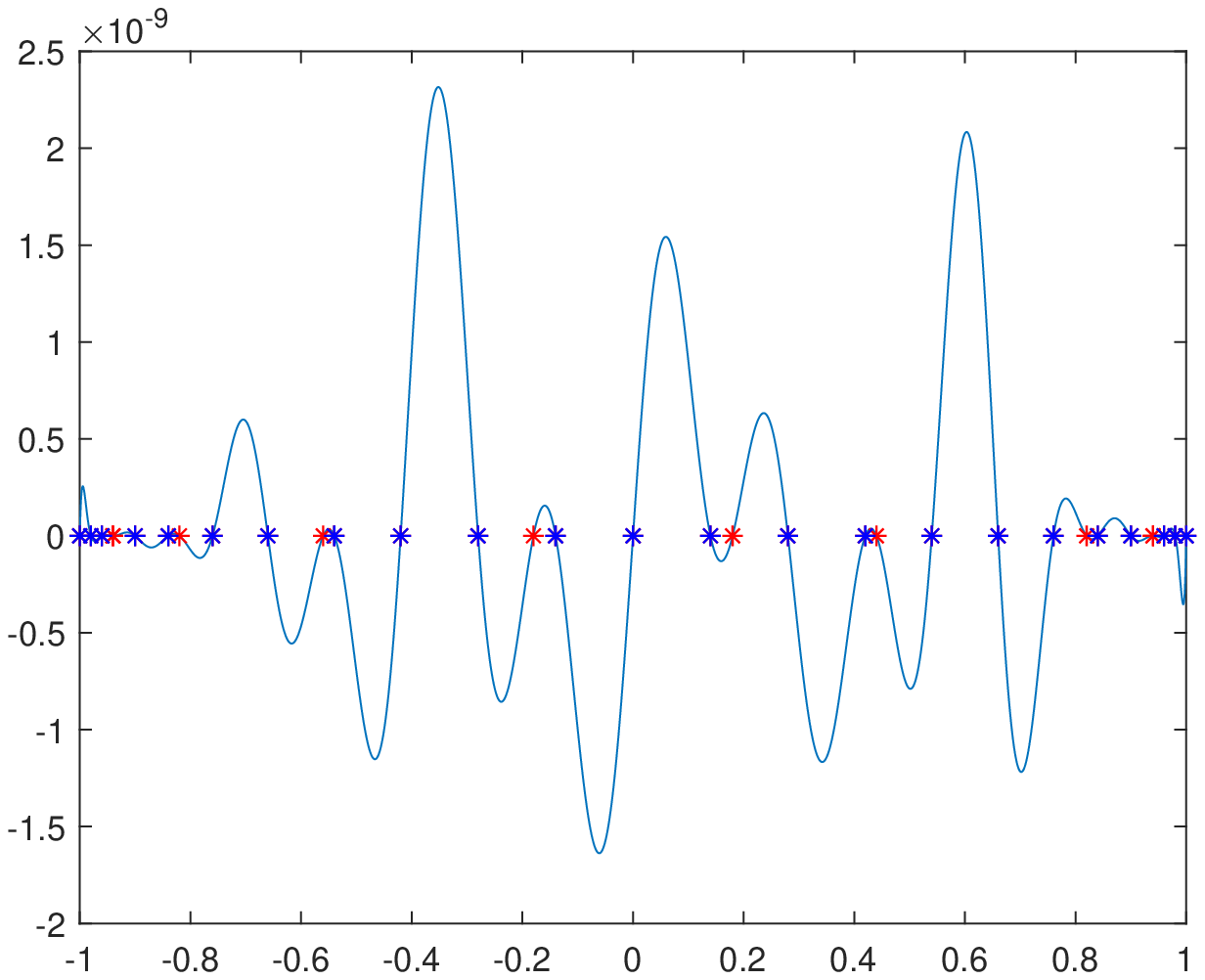} 
       \includegraphics[width=.32\textwidth]{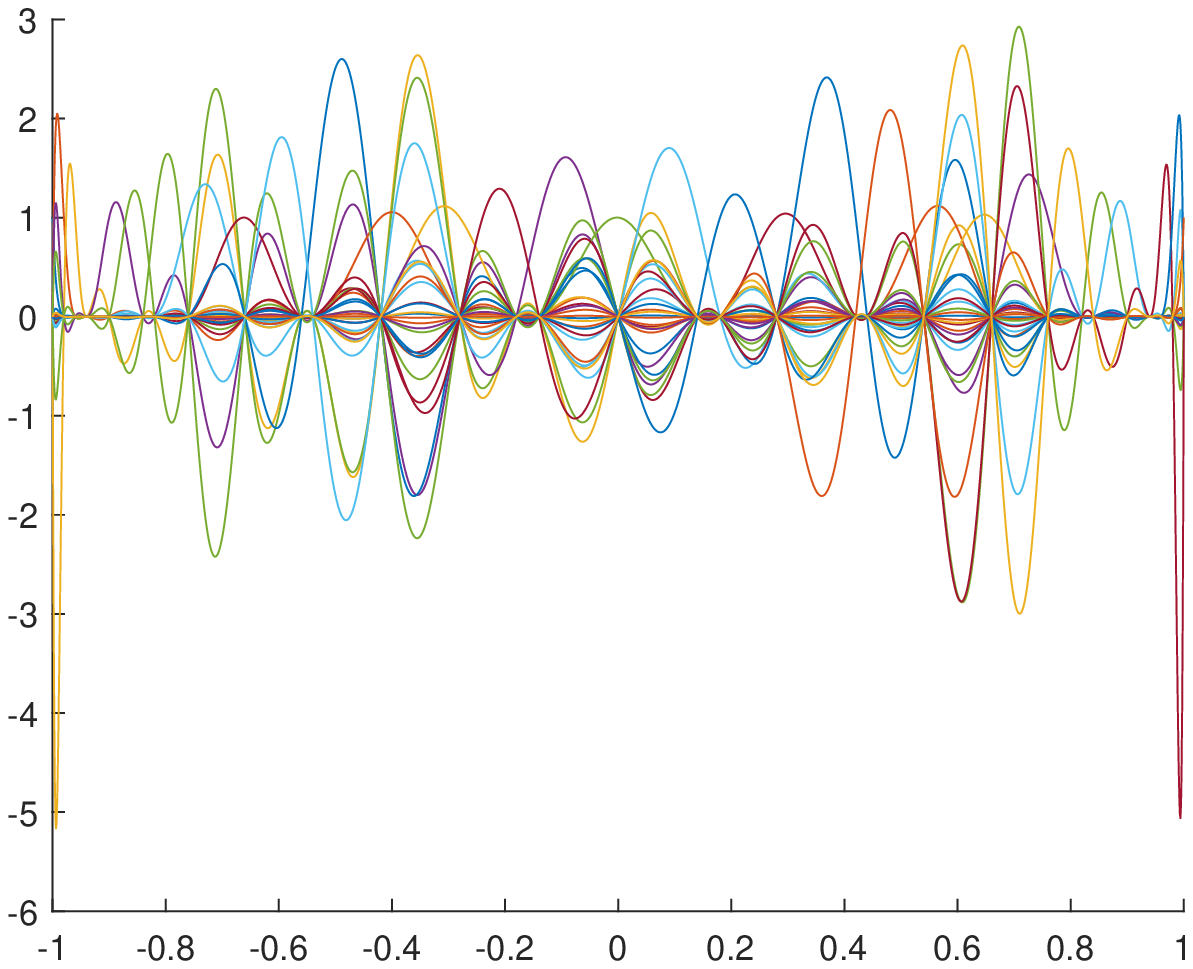} 
         \caption{The nodal polynomial at the mock-Chebyshev nodes $X^{\prime}_m$ (left), the nodal polynomial at the node set $X^{\prime}_m\cup X^{\prime\prime\prime}_p$(center) and the Lagrange fundamental polynomials at the node set $X^{\prime}_m\cup X^{\prime\prime\prime}_p$(right) for $n=100$ (and then $m=22$, $p=9$, $r=32$). }
    \label{Fig:Casen100}
    \end{center}
\end{figure}

 \begin{figure} 
 \begin{center}
          \includegraphics[width=.32\textwidth]{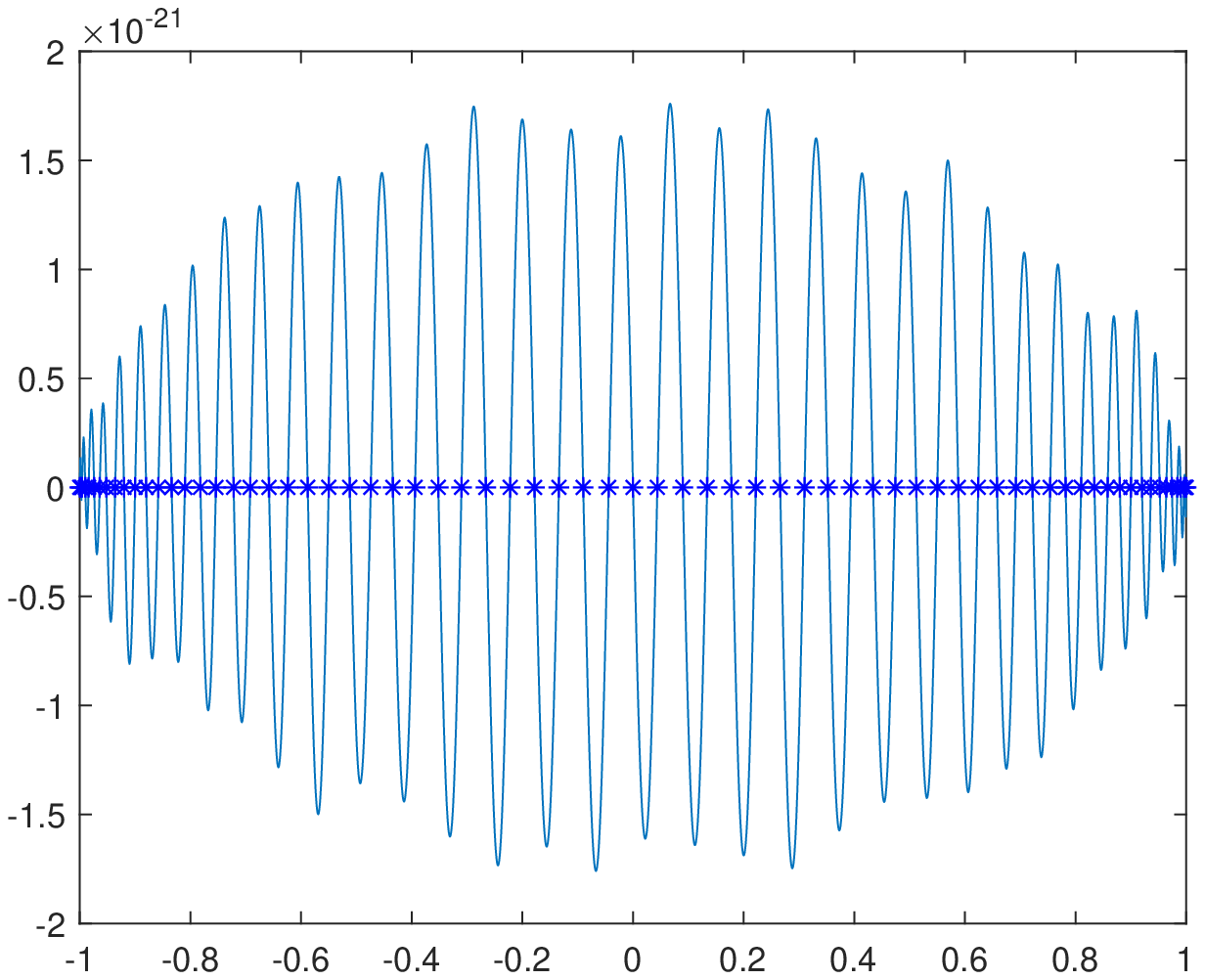}
    \includegraphics[width=.32\textwidth]{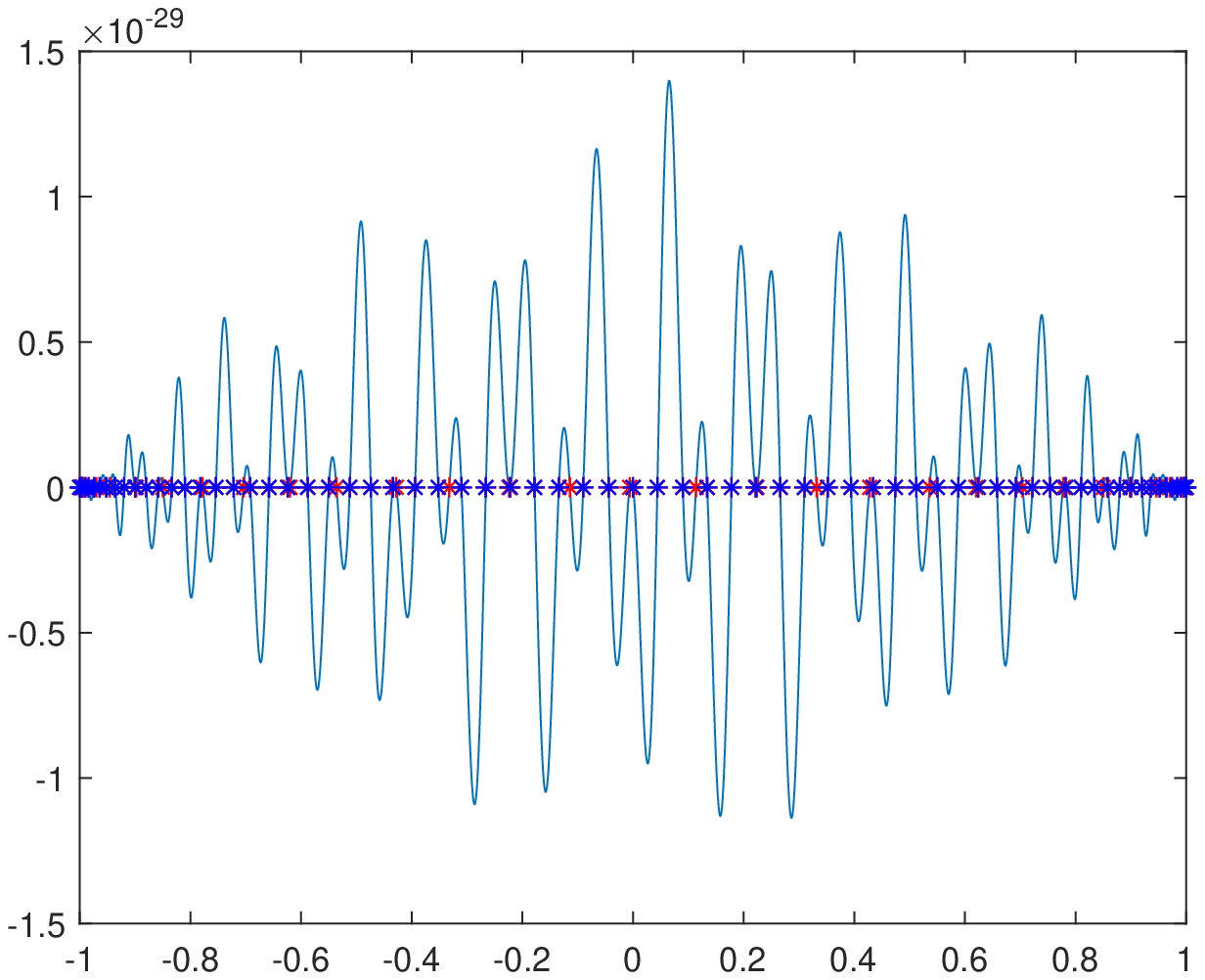} 
    \includegraphics[width=.32\textwidth]{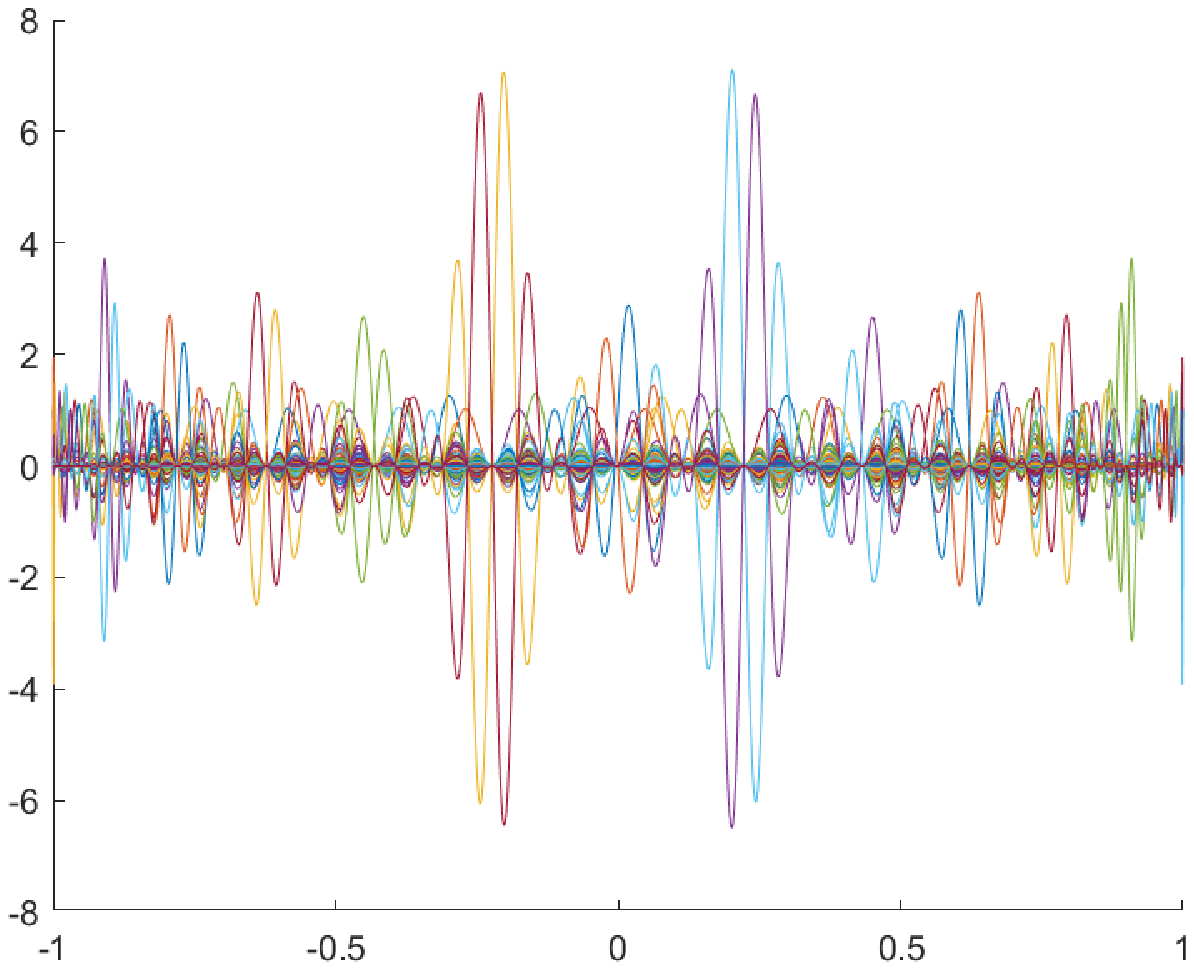}   
         \caption{The nodal polynomial at the mock-Chebyshev nodes $X^{\prime}_m$ (left), the nodal polynomial at the node set $X^{\prime}_m\cup X^{\prime\prime\prime}_p$(center) and the Lagrange fundamental polynomials at the node set $X^{\prime}_m\cup X^{\prime\prime\prime}_p$(right) for $n=1000$, (and then $m=70$, $p=28$, $r=99$). }
    \label{Fig:Casen1000}
    \end{center}
\end{figure}

 \begin{figure} 
 \begin{center}
          \includegraphics[width=.32\textwidth]{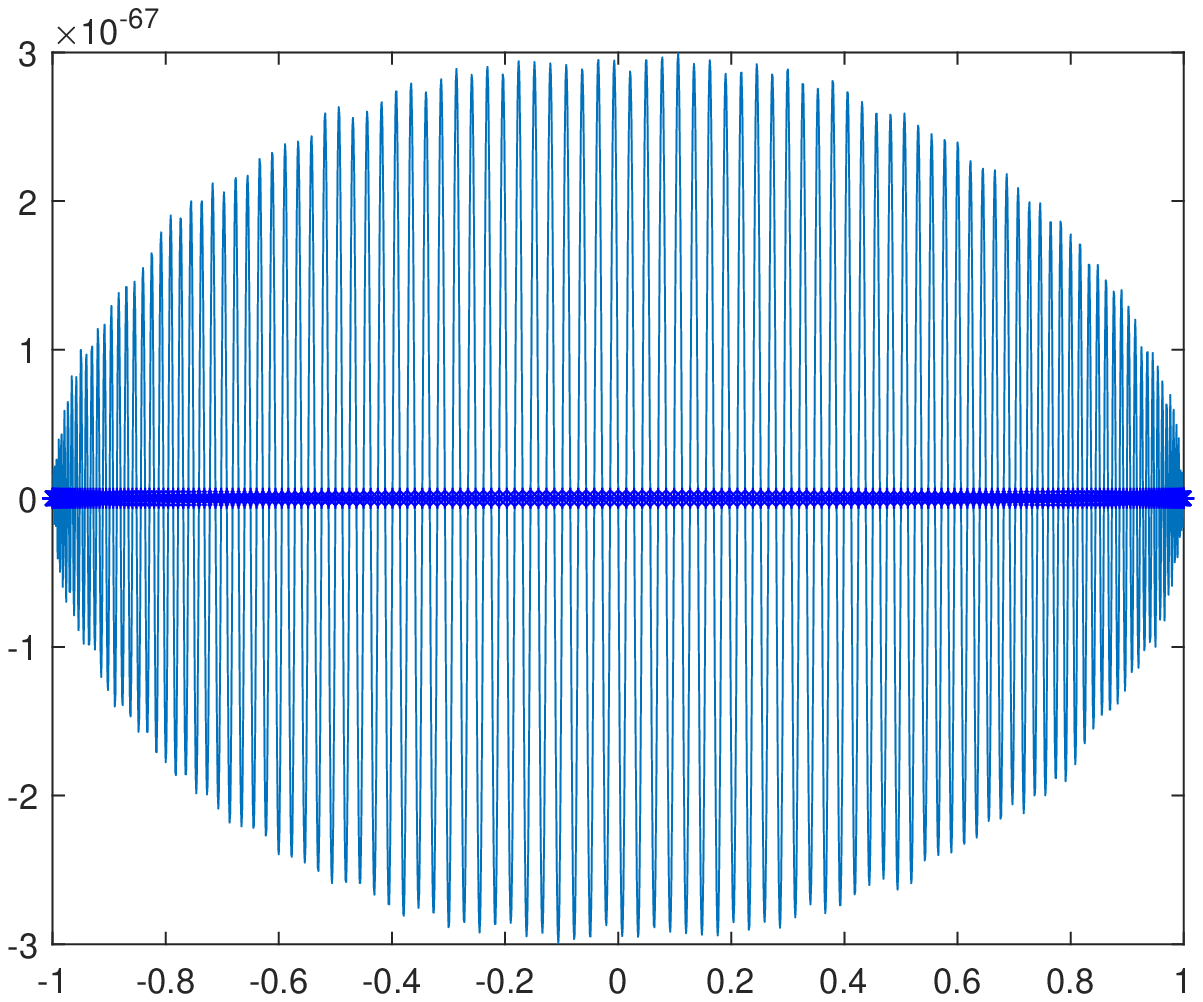}
    \includegraphics[width=.32\textwidth]{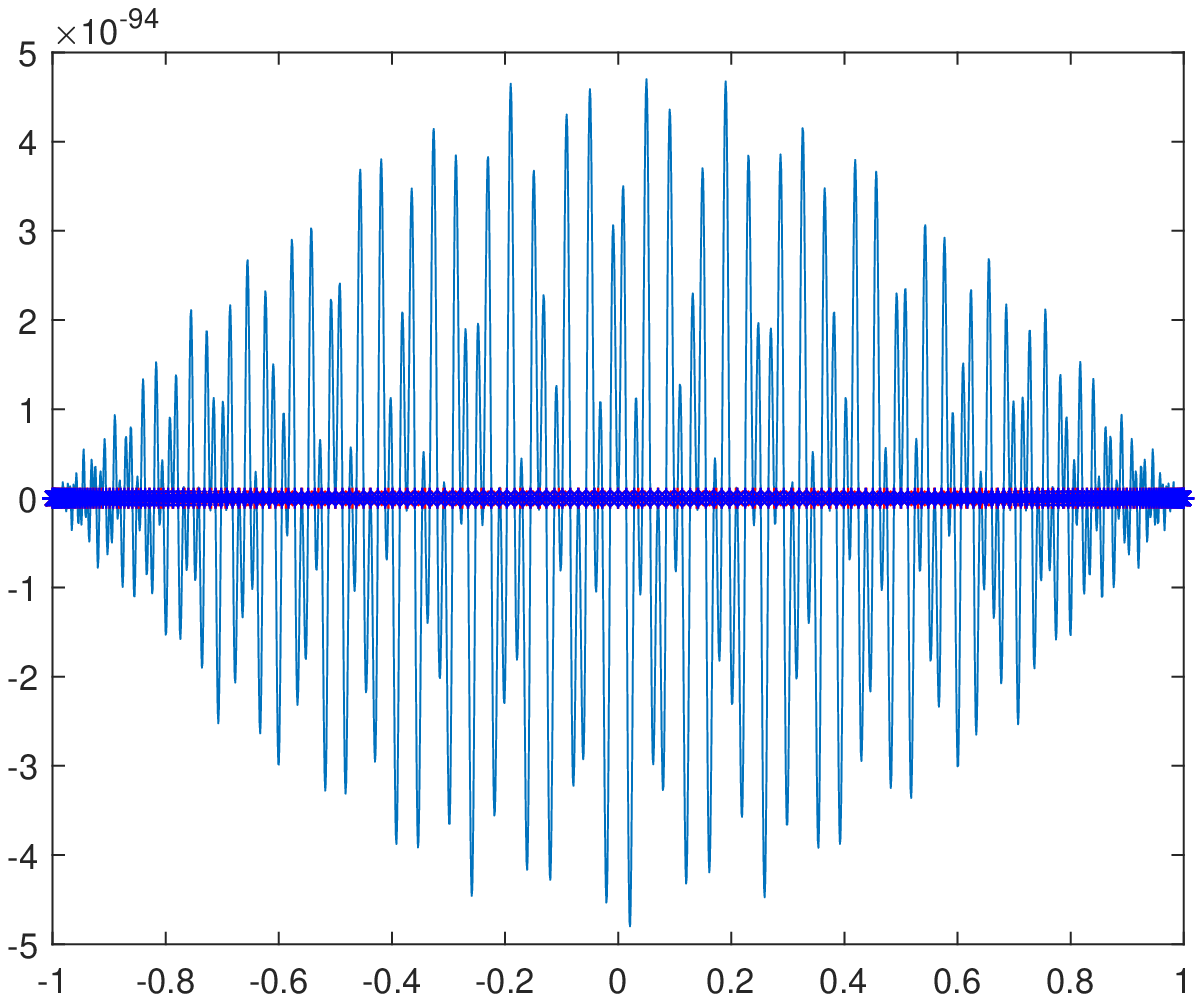} 
    \includegraphics[width=.32\textwidth]{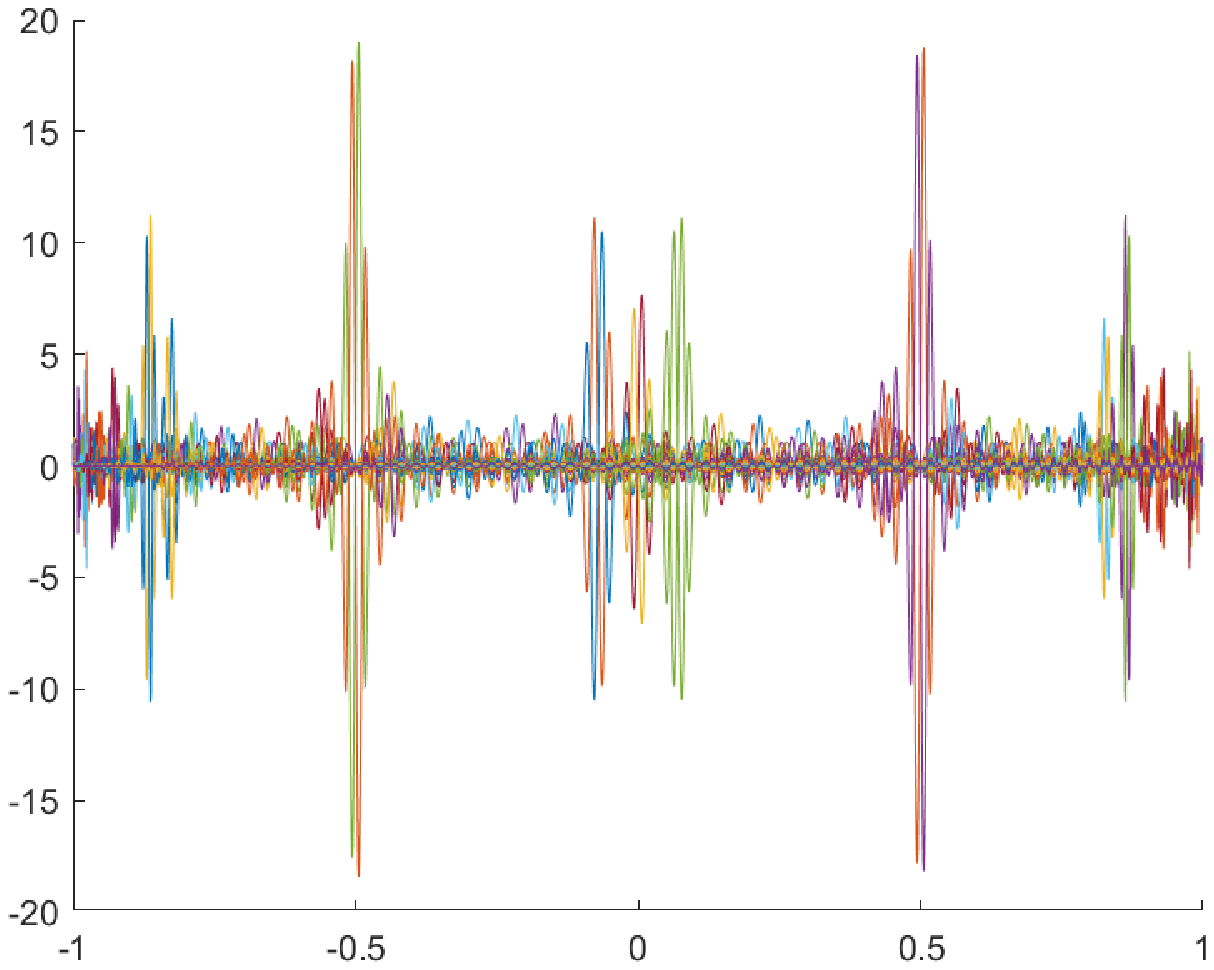}   
         \caption{The nodal polynomial at the mock-Chebyshev nodes $X^{\prime}_m$ (left), the nodal polynomial at the node set $X^{\prime}_m\cup X^{\prime\prime\prime}_p$(center) and the Lagrange fundamental polynomials at the node set $X^{\prime}_m\cup X^{\prime\prime\prime}_p$(right) for $n=10000$, (and then $m=222$, $p=90$, $r=313$). }
    \label{Fig:Casen10000}
    \end{center}
\end{figure}

\section{Numerical Differentiation through constrained mock-Chebyshev least squares operator}
\label{Sec:NumericalDifferentiation }
In this Section we introduce a numerical differentiation formula based on the constrained mock-Chebyshev least squares operator. Let $f\in C^1([-1,1])$ be a differentiable function whose first derivative is continuous on the interval $(-1,1)$. It is worth emphasizing that we are supposing to know exclusively the evaluations of the function $f$ on the set $X_n$ of $n+1$ equispaced nodes. We apply the constrained mock-Chebyshev least squares operator to the function $f$ and we compute the polynomial
\begin{equation}
\label{eq:Constrained_Mock_Cheb_Appr}
    \hat{P}_{r,n}[f](x)=\sum_{i=0}^r \hat{a}_i u_i(x),    \quad x\in[-1,1].
\end{equation}
By differentiating~\eqref{eq:Constrained_Mock_Cheb_Appr}, we obtain 
\begin{equation}
\label{eq:Constrained_Mock_Cheb_ApprDeriv}
    \hat{P}_{r,n}^{\prime}[f](x)=\sum_{i=0}^r \hat{a}_i u_i^{\prime}(x), \quad x\in[-1,1],
\end{equation}
and, since $\hat{P}_{r,n}^{\prime}[f]\in\Pi_{r-1}$, from~\eqref{eq:PropReprodPolyn}, we get
\begin{equation}
\label{eq:GlobPolAppDerivative}
   \hat{P}_{r,n}^{\prime}[f]=\hat{P}_{r,n}\left[ \hat{P}_{r,n}^{\prime}[f]\right]=\sum_{i=0}^r \hat{a}^{\prime}_i u_i(x).
\end{equation}
\begin{remark}
We notice that the coefficients vector $\hat{a}^{\prime}=[\hat{a}^{\prime}_0,\dots,\hat{a}^{\prime}_r]^T$ is  uniquely determined, since the KKT matrix~\eqref{KKTMatrix} has nonzero determinant, see~\cite{DellAccio:2022:GOT}.
\end{remark}

Relations between the vectors of coefficients $\hat{a}=[\hat{a}_0,\dots,\hat{a}_r]^T$ and  $\hat{a}^{\prime}=[\hat{a}^{\prime}_0,\dots,\hat{a}^{\prime}_r]^T$ depends on the chosen polynomial basis $\{u_1(x),\dots,u_r(x)\}$. Since we are assuming to work with the Chebyshev polynomial basis of the first kind $\mathcal{B}_{C,1}$, in the following we make this relation explicit in this particular case. To this aim we recall some useful identities between the Chebyshev polynomials of first kind $T_k(x)$, $k=0,\dots, r$, and the Chebyshev polynomial of the second kind $U_k(x)$, $k=0,\dots,r$. This polynomials are defined by~\cite{Rivlin:2020CP}
$$U_k(x)=\frac{\sin\left((k+1)\arccos x\right)}{\sin\left(\arccos x \right)}, \quad x\in[-1,1], \quad k=0,\dots, r, $$
and satisfy the following relations
\begin{equation}
\label{eq:RelFirstSecondChebPolFI}
   U_0(x)=T_0(x)=1, \quad U_1(x)=2T_1(x), \quad U_k(x)-U_{k-2}(x)=2T_k(x), \quad k\ge2.
\end{equation}
Moreover 
\begin{equation}
\label{eq:RelFirstSecondChebPol}
   T_k^{\prime}(x)=kU_{k-1}(x), \quad U_k^{\prime}(x)=\frac{(k+1)T_{k+1}(x)-xU_k(x)}{x^2-1}.
\end{equation}

\begin{theorem}
\label{Thm:RelCoefficients}
Let be $f\in C([-1,1])$, by expressing the polynomial $\hat{P}_{r,n}[f]$ and its first derivative $\hat{P}_{r,n}^{\prime}[f]$ in the basis $\mathcal{B}_{C,1}$ as in equations~\eqref{eq:Constrained_Mock_Cheb_Appr} and~\eqref{eq:GlobPolAppDerivative}, respectively, we have 
\begin{equation}
\label{relVectorOfCoeff}
 \hat{a}^{\prime}_0= \sum_{j=0}^{\left\lfloor \frac{r}{2}\right\rfloor}(2j+1)\hat{a}_{2j+1}, \qquad    \hat{a}^{\prime}_i=2\sum_{j=0}^{\left\lfloor \frac{r-i}{2}\right\rfloor}(i+2j+1)\hat{a}_{i+2j+1}, \quad i=1,\dots,r.
\end{equation}
\end{theorem}
\begin{proof}
By~\eqref{eq:Constrained_Mock_Cheb_Appr}, the polynomial $\hat{P}_{r,n}[f]$ in the basis $\mathcal{B}_{C,1}$ has the form
\begin{equation*}
    \hat{P}_{r,n}[f](x)=\sum_{j=0}^r \hat{a}_j T_j(x),    \quad x\in[-1,1].
\end{equation*}
By using the identities ~\eqref{eq:Constrained_Mock_Cheb_ApprDeriv} and~\eqref{eq:RelFirstSecondChebPol}, we get
\begin{equation}
    \hat{P}_{r,n}^{\prime}[f](x)=\sum_{j=0}^r \hat{a}_j T^{\prime}_j(x)=\sum_{j=1}^r j\hat{a}_j U_{j-1}(x),    \quad x\in[-1,1].
\end{equation}
By setting $\hat{a}_{r+1}=0$, after a change the dummy index, the polynomial $\hat{P}_{r,n}^{\prime}[f]$ can be written as 
\begin{equation}
    \hat{P}_{r,n}^{\prime}[f](x)=\sum_{j=0}^r (j+1)\hat{a}_{j+1}U_{j}(x),    \quad x\in[-1,1].
\end{equation}
The result follows from the identity~\eqref{eq:RelFirstSecondChebPolFI}.
\end{proof}

From the above results, one can deduce that there are two different strategies in order to compute the analytic expression~\eqref{eq:GlobPolAppDerivative} of the polynomial $\hat{P}^{\prime}_{r,n}[f]$ in the Chebyshev polynomial basis of the first kind. 
\begin{itemize}
    \item[$S_1$)] Evaluate $U_k(x)$, $k=0,\dots,r$, on the set $X_n$. Compute the values of $\hat{P}^{\prime}_{r,n}[f]$ on the equispaced nodes using~\eqref{eq:Constrained_Mock_Cheb_ApprDeriv}. Solve the KKT linear equations~\eqref{Cons_Least_Squares_Prob} in order to compute $\hat{P}_{r,n}\left[\hat{P}^{\prime}_{r,n}[f]\right]$ from the values of $\hat{P}^{\prime}_{r,n}[f]$ on the equispaced nodes. 
     \item[$S_2$)] Use the equation~\eqref{relVectorOfCoeff} in order to compute the analytic expression of $\hat{P}^{\prime}_{r,n}[f]$.
\end{itemize}
We notice that, although the strategy $S_2$ is more direct with respect to the strategy $S_1$, it can be applied only in the case of the Chebyshev polynomial basis of the first kind. In the next Section, we will show that the two strategies are equivalent in terms of accuracy of results.

We emphasize, that formula~\eqref{eq:GlobPolAppDerivative}  provides a global polynomial approximation of the first derivative of the function $f$.
 Clearly, it is possible to repeat both procedures to approximate the derivative of order $k$, for $k\ge1$, by supposing that $f\in C^k([-1,1])$. In this regard, the following Theorem holds.
\begin{theorem}
\label{Thm:RelCoefficientsk}
Let be $f\in C([-1,1])$. We express the polynomial $\hat{P}_{r,n}[f]$ and its successive derivatives in the basis $\mathcal{B}_{C,1}$, that is
\begin{equation*}
    \hat{P}_{r,n}^{(\nu)}[f](x)=\sum_{i=0}^r \hat{a}_i^{(\nu)}T_i(x), \quad x\in[-1,1], \quad \nu=1,\dots,r.
\end{equation*}
For each $\nu\ge 1$, we get
\begin{equation}
\label{relVectorOfCoeffkk}
 \hat{a}^{(\nu)}_0= \sum_{j=0}^{\left\lfloor \frac{r}{2}\right\rfloor}(2j+1)\hat{a}^{(\nu-1)}_{2j+1}, \qquad    \hat{a}^{(\nu)}_i=2\sum_{j=0}^{\left\lfloor \frac{r-i}{2}\right\rfloor}(i+2j+1)\hat{a}^{(\nu-1)}_{i+2j+1}, \quad i=1,\dots,r.
\end{equation}
\end{theorem}
\begin{proof}
The proof follows the same argument of Theorem~\ref{Thm:RelCoefficients}.
It is therefore omitted here.
\end{proof}

\section{Numerical experiments}
\label{Sec:NumericalExp}
In this Section, we numerically prove the accuracy of the proposed method by several examples. The numerical experiments are performed using \texttt{MatLab} software. In particular, the command \texttt{derivative} is used in order to compute the exact successive derivatives of all considered functions and the Chebfun package is used in order to compute the Chebyshev polynomial basis of the first kind~\cite{Hale:2013:FAA}.

We perform two different types of numerical tests. In the first test, we consider the function 
\begin{equation*}
    f_1(x)=xe^{-2x}+\sin(3x)
\end{equation*}
used in~\cite{Li:2005:GED} in order to test general explicit finite difference formulas with arbitrary order accuracy for approximating first and higher derivatives. These formulas are applicable to unequally or equally spaced data. In line with the experiments presented in~\cite{Li:2005:GED}, we consider a set of $n+1=67$ equispaced points in the interval $[-1,1]$, in order to have a stepsize $h=0.03$.
We compute the errors
\begin{equation}
\label{eq:Mean_approx_error}
  e_{mean}:=\frac{1}{N}\sum_{i=1}^N e_i, \qquad e_{max}:=\max_{i=1,\dots,N} e_i,
\end{equation}
obtained in approximating
the first four order derivatives of the function $f_1$ by the constrained mock-Chebyshev least squares operator on the uniform grid of $67$ points in $[-1,1]$, computed by following the strategy $S_2$. In equation~\eqref{eq:Mean_approx_error} $e_i$ is the absolute value of the approximation error at the $i$-th point of this grid. The numerical results are reported in Table~\ref{tab:InAnalogytoJinping}. The approximation accuracies are comparable or even better with respect to those one reported in~\cite{Li:2005:GED} for the case of the finite difference formula at $11$ equally spaced points with  stepsize $h=0.03$ in the interval $[0,0.3]$. To better appreciate the behavior of the approximation errors in the whole interval $[-1,1]$, in Figure~\ref{Fig_Eerrortrend} we plot the absolute values of the pointwise errors computed on the equispaced grid of $N=201$ points for the first four order derivatives of the function $f_1$.  The plots are displayed in a lexicographic order, by increasing the order of derivatives. The red  dash-dotted line is the error of approximation related to the application of the strategy $S_1$ while the black dashed line  is the error of approximation related to the application of the strategy $S_2$. From the Figure, it is evident that the application of the two strategies $S_1$ and $S_2$ gives practically the same results.

\begin{table}
\begin{center}
\label{tab:InAnalogytoJinping}%
\begin{tabular}{cccccc}
 $order$  & 0 & 1  & 2 & 3 & 4  \\
 \hline
  $e_{mean}$ & 1.24e-15 &  7.59-14 & 9.02-12 &   9.92e-10 &  8.57e-08\\
 $e_{max}$ &  1.77e-14 & 4.43e-12  & 7.46e-10 & 7.67e-08 &    5.78e-06  \\
\end{tabular}
\caption{Mean and max approximation errors obtained in approximating the function $f_1(x)=xe^{-2x}+\sin(3x)$ and its first four order derivatives on the grid of $67$ equispaced points in $[-1,1]$, by using the constrained mock-Chebyshev least squares operator with $n+1=67$. }
\end{center}
\end{table}

\begin{figure} 
\begin{center}
    \includegraphics[width=.49\textwidth]{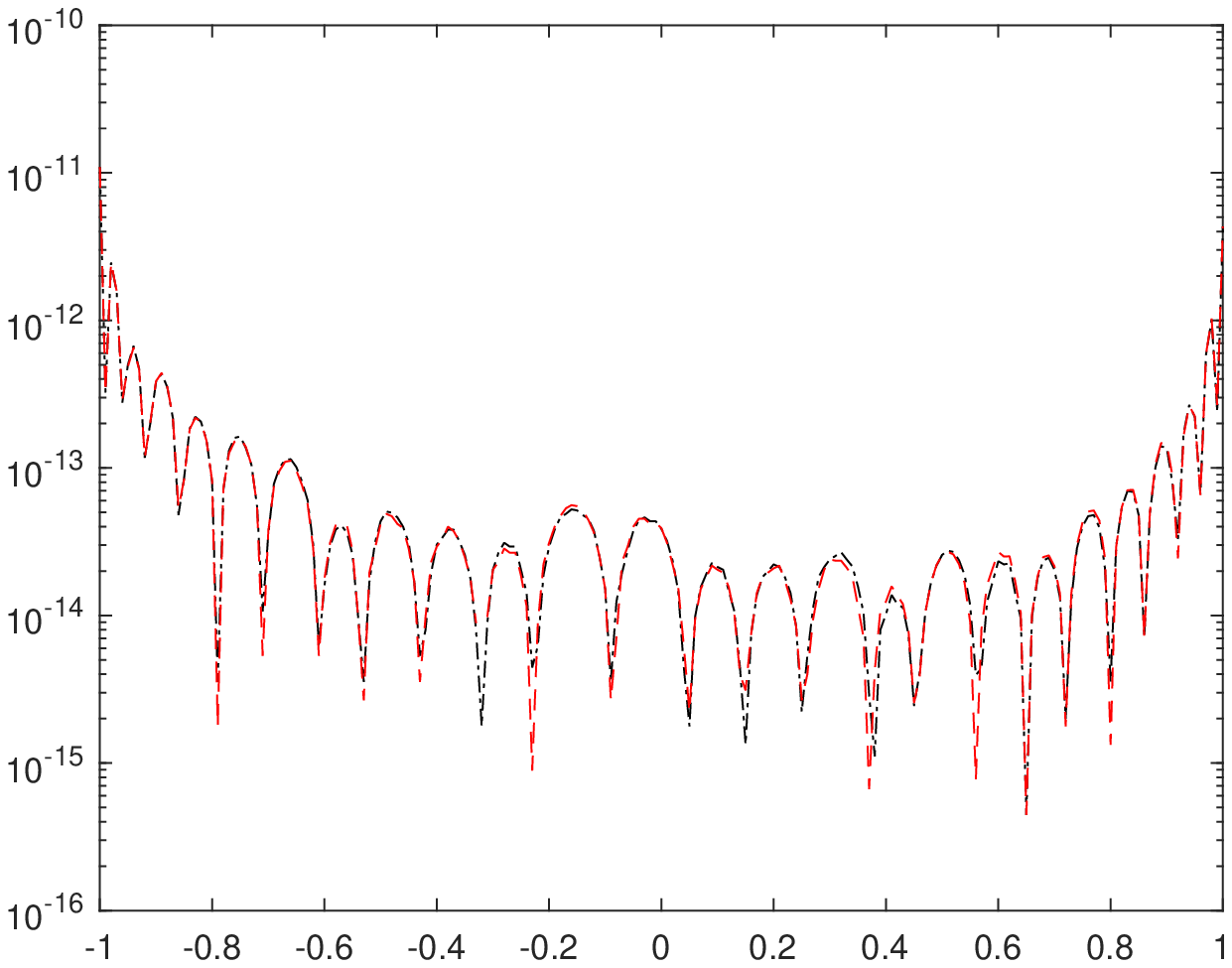} 
        \includegraphics[width=.49\textwidth]{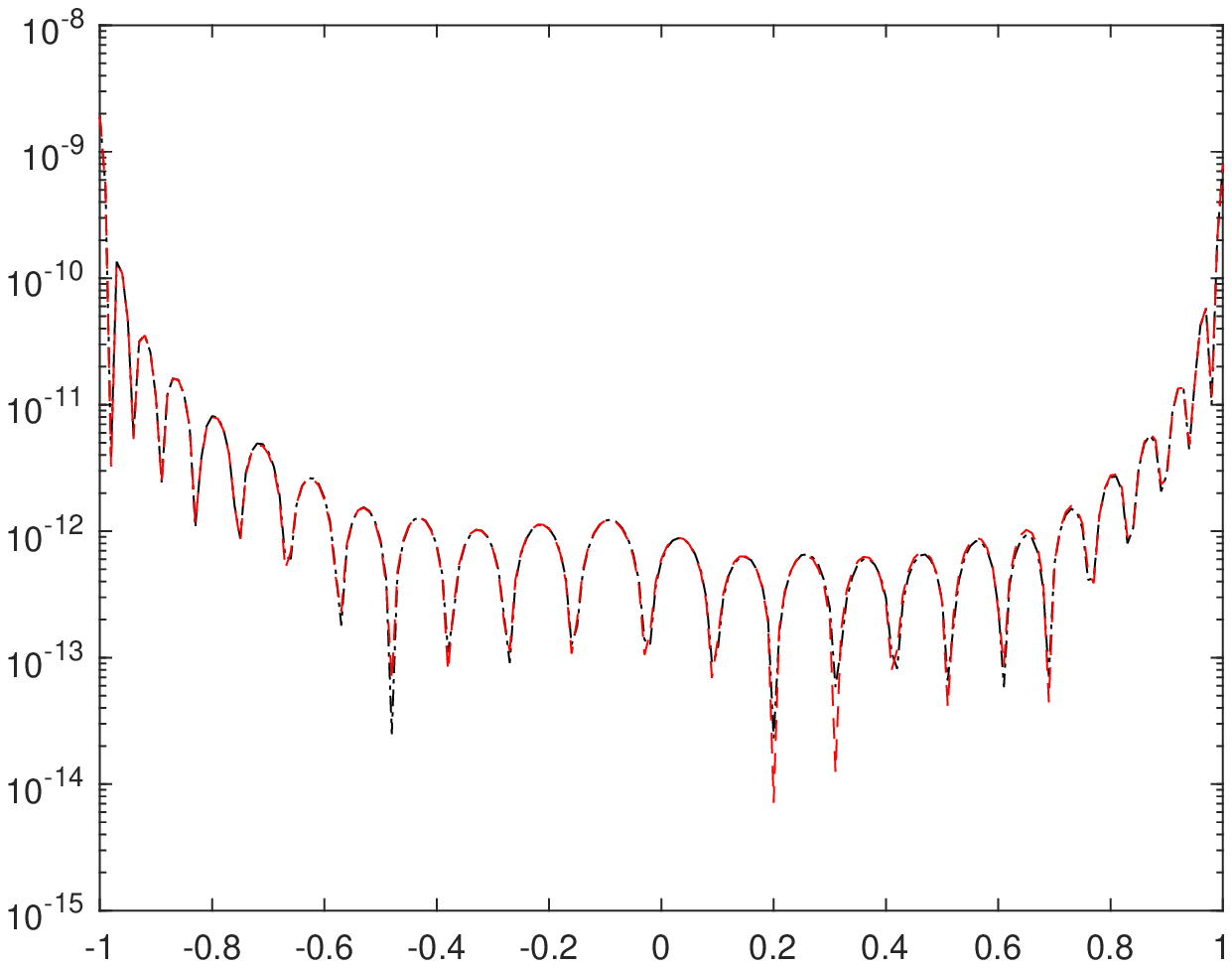}
            \includegraphics[width=.49\textwidth]{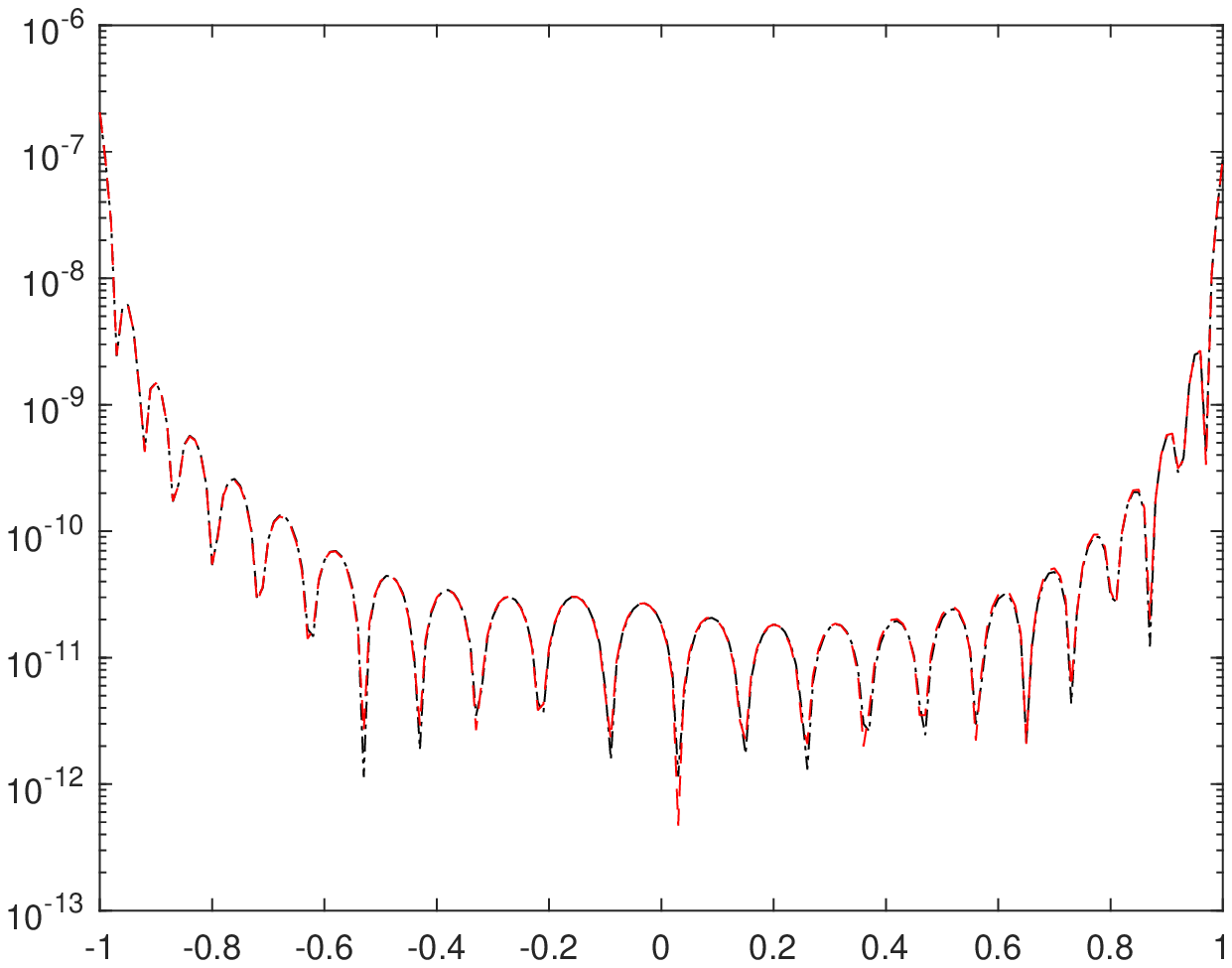}
                \includegraphics[width=.49\textwidth]{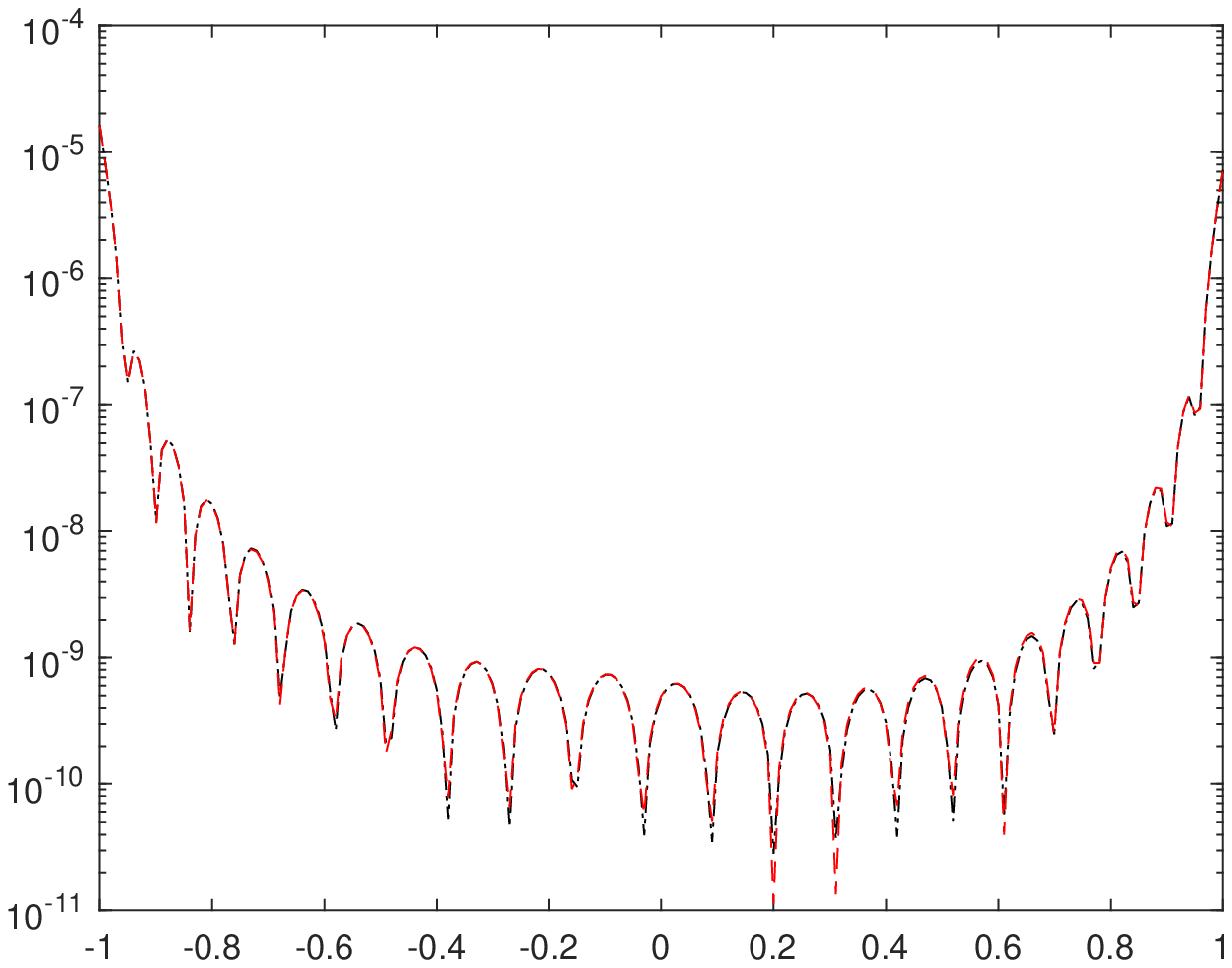} 
         \caption{Behavior of the approximation errors, in absolute value, of the first four order derivatives of the function  $f_1(x)=xe^{-2x}+\sin(3x)$ in the whole interval $[-1,1]$, by using the constrained mock-Chebyshev least squares operator with $n+1=67$.  The absolute values of the pointwise errors are computed on the equispaced grid of $N=201$ points. The plots are displayed in a lexicographic order, by increasing the order of derivatives. From the plots, it is evident that the application of the two strategies $S_1$ (red  dash-dotted line) and $S_2$ (black dashed line) gives practically the same results.}
         \label{Fig_Eerrortrend}
\end{center}
\end{figure}

In the second type of test, we consider the following functions~\cite{Li:2005:GED, Klein:2012:LRF}
\begin{eqnarray*}
 && f_1(x)=xe^{-2x}+\sin(3x), \quad f_2(x)=e^{-50(x-0.4)^2}+\sinh(x),\\
&& f_3(x)=\frac{1}{1+8x^2}, \quad f_4(x)=\frac{1}{1+25x^2}, 
\end{eqnarray*}
and we analyze the trend of the mean approximation errors and the maximum approximation errors~\eqref{eq:Mean_approx_error} obtained in approximating the first four order derivatives of $f_1-f_4$ by using the constrained mock-Chebyshev least squares operator on uniform grids of different stepsize. In particular, we consider sets of $n+1$ equispaced nodes with $n=50k$,  $k=1,\dots,80$ and compute the errors on a grid of $N=10^5$ random points of the interval $[-1,1]$.

 The results of the tests are shown in Figures~\ref{figureErrorg},~\ref{figureErrorf8},~\ref{figureErrorf1}~\ref{figureErrorf4}. All these examples show, with clear evidence, that once the maximum precision is reached for $\hat{P}_{r,n}[f]$, the increase in the number of nodes does not lead to more accurate approximation for the derivatives, on the contrary, the increase of the condition number of the matrices involved in the strategies $S_1$ and $S_2$ causes worsening of results.

 \begin{figure} 
 \begin{center}
          \includegraphics[width=.49\textwidth]{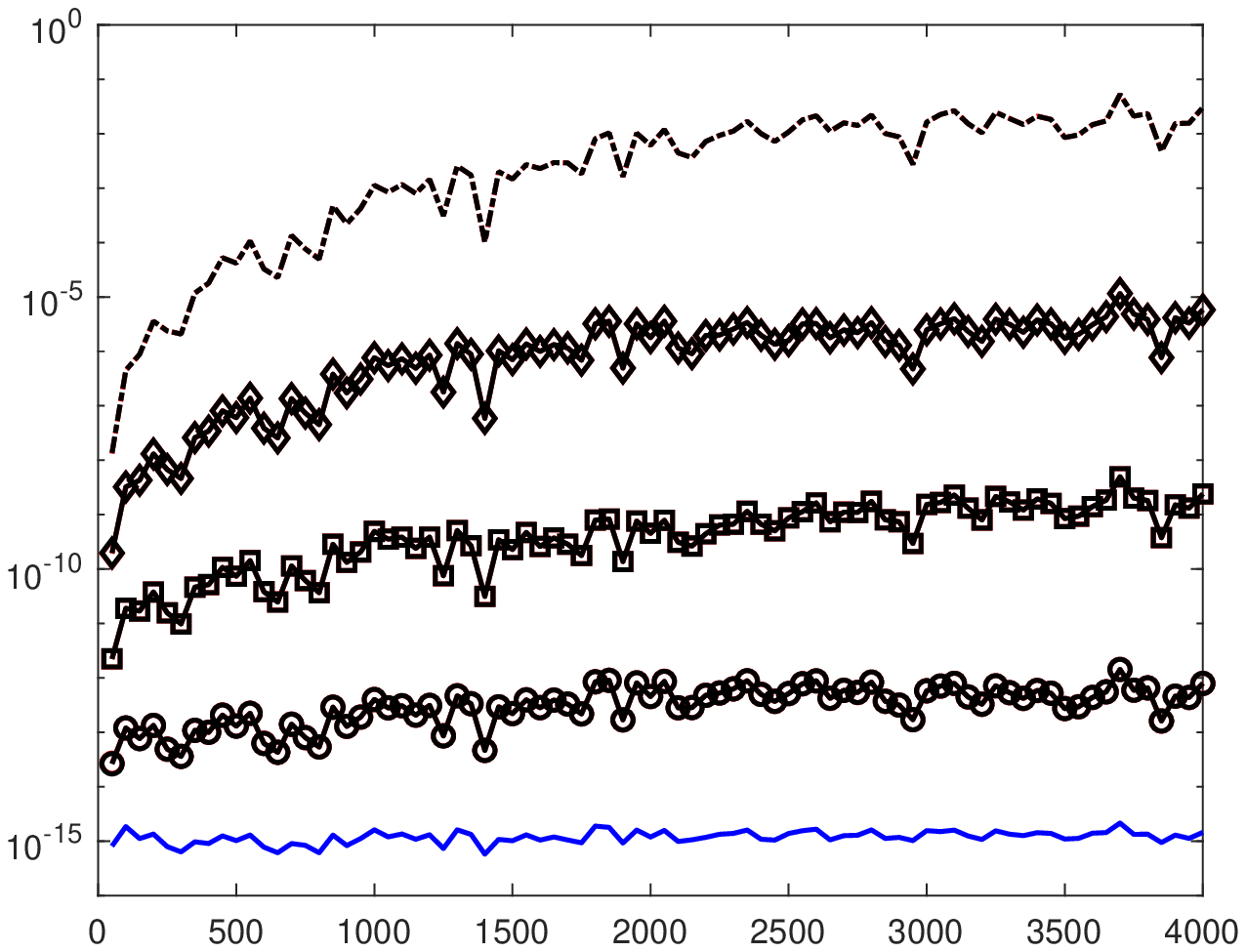}
          \includegraphics[width=.49\textwidth]{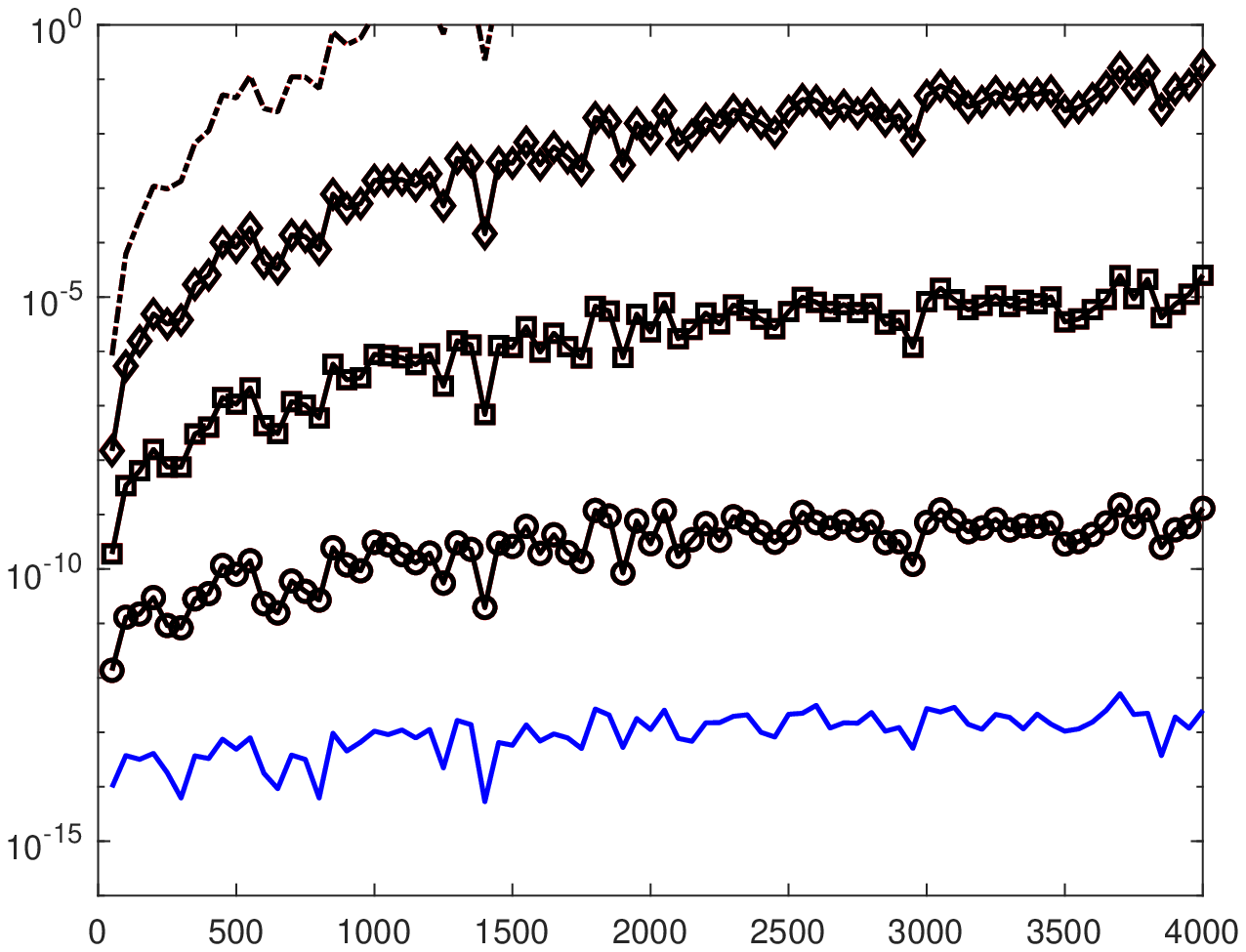}
        \caption{Mean approximation error (left) and Maximum approximation error (right) at $N=100000$ random points in the interval $[-1,1]$ relative to the function $f_1(x)=xe^{-2x}+\sin(3x)$ (in blue) and its first four derivatives, with the increasing order from the bottom to the top, when approximated by using the constrained mock-Chebyshev least squares operator with $n=50k$, $k=1,\dots,80$.  From the plots, it is evident that the application of the two strategies $S_1$ (red) and $S_2$ (black) gives practically the same results.}
    \label{figureErrorg}
    \end{center}
\end{figure}

\begin{figure}      
\begin{center}
\includegraphics[width=.49\textwidth]{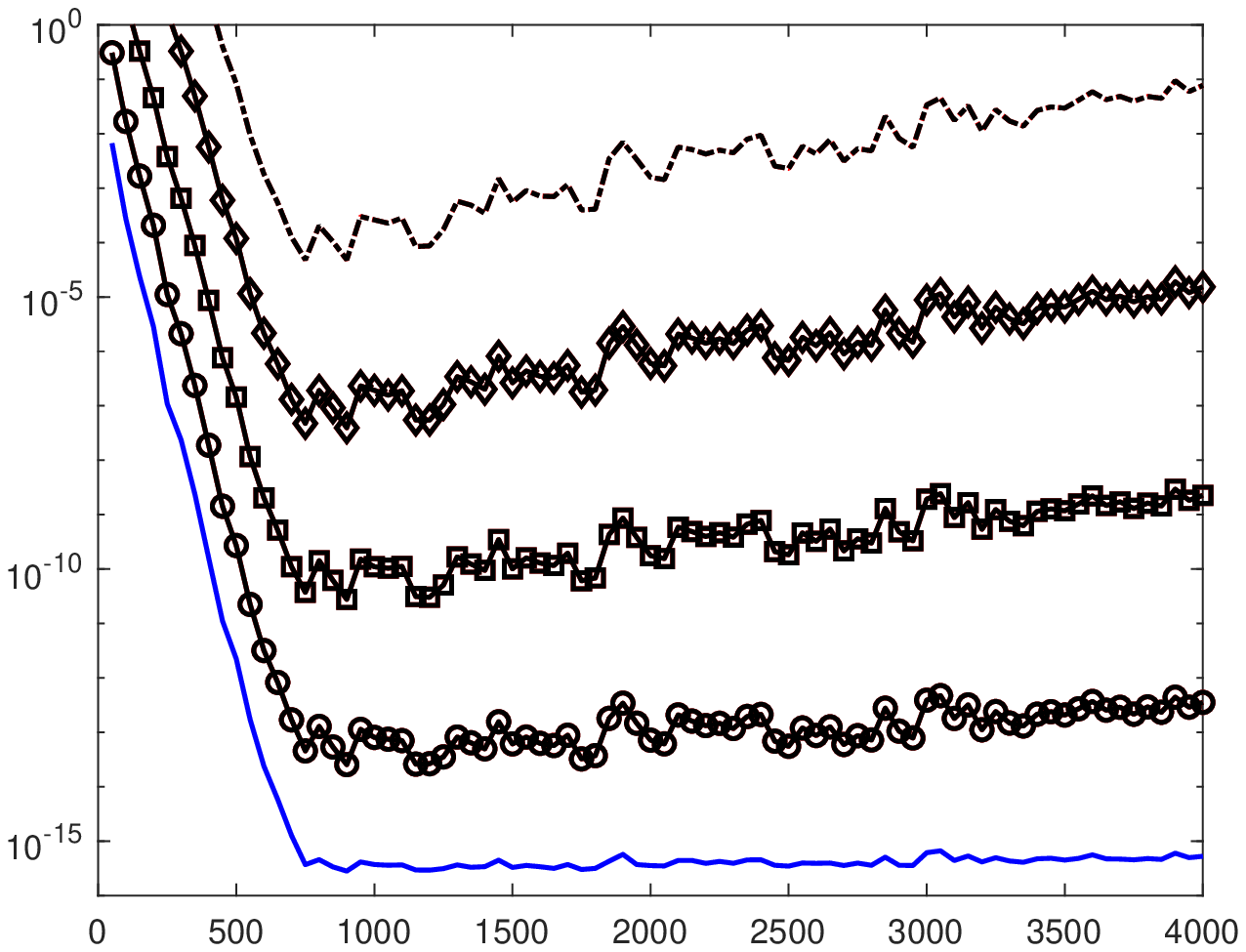}
  \includegraphics[width=.49\textwidth]{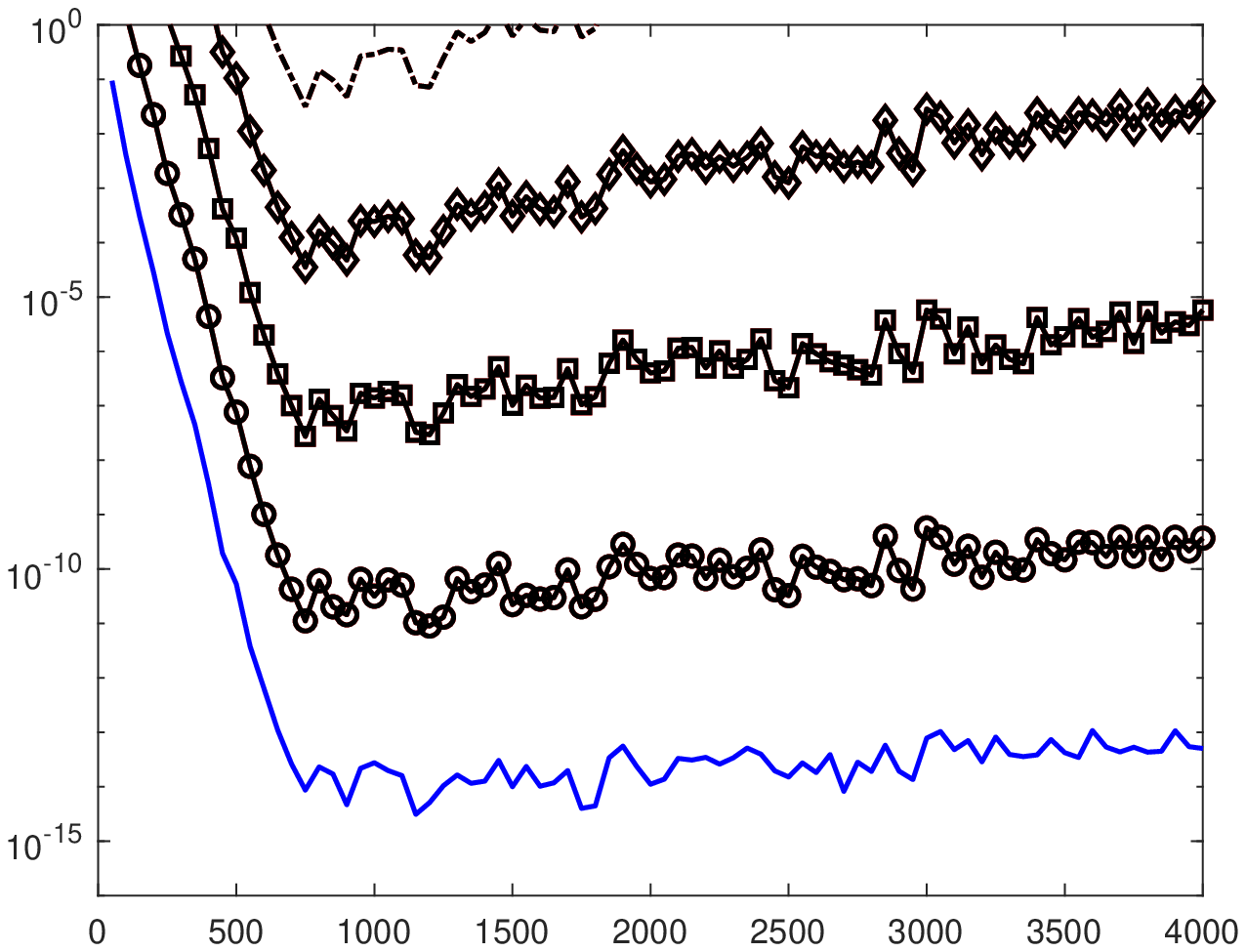}
         \caption{Mean approximation error (left) and Maximum approximation error (right) at $N=100000$ random points in the interval $[-1,1]$ relative to the function $ f_2(x)=e^{-50(x-0.4)^2}+\sinh(x)$ (in blue) and its first four derivatives, with the increasing order from the bottom to the top, when approximated by using the constrained mock-Chebyshev least squares operator with $n=50k$, $k=1,\dots,80$.  From the plots, it is evident that the application of the two strategies $S_1$ (red) and $S_2$ (black) gives practically the same results.}
    \label{figureErrorf8}
    \end{center}
\end{figure}

\begin{figure}      
\begin{center}
\includegraphics[width=.49\textwidth]{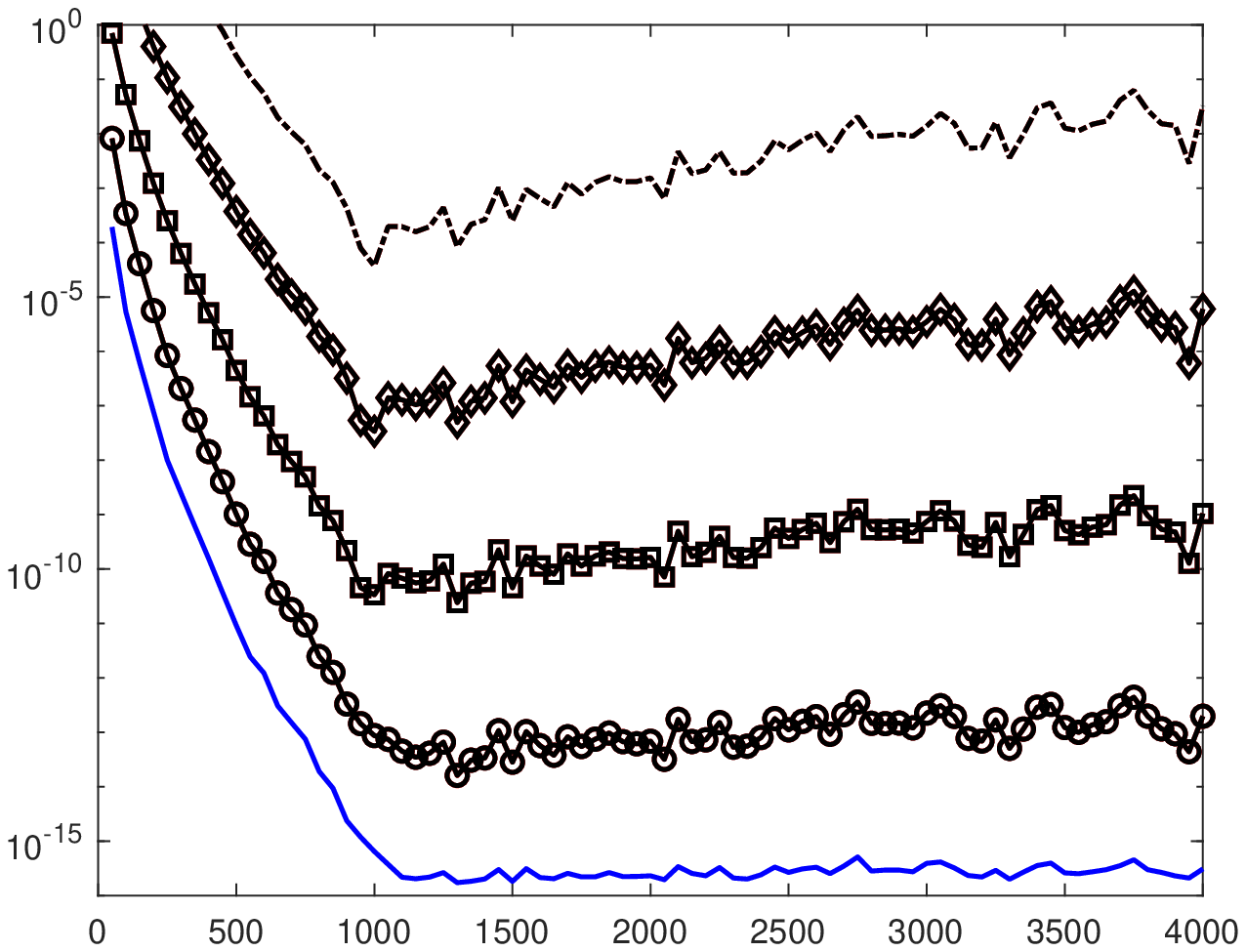}
  \includegraphics[width=.49\textwidth]{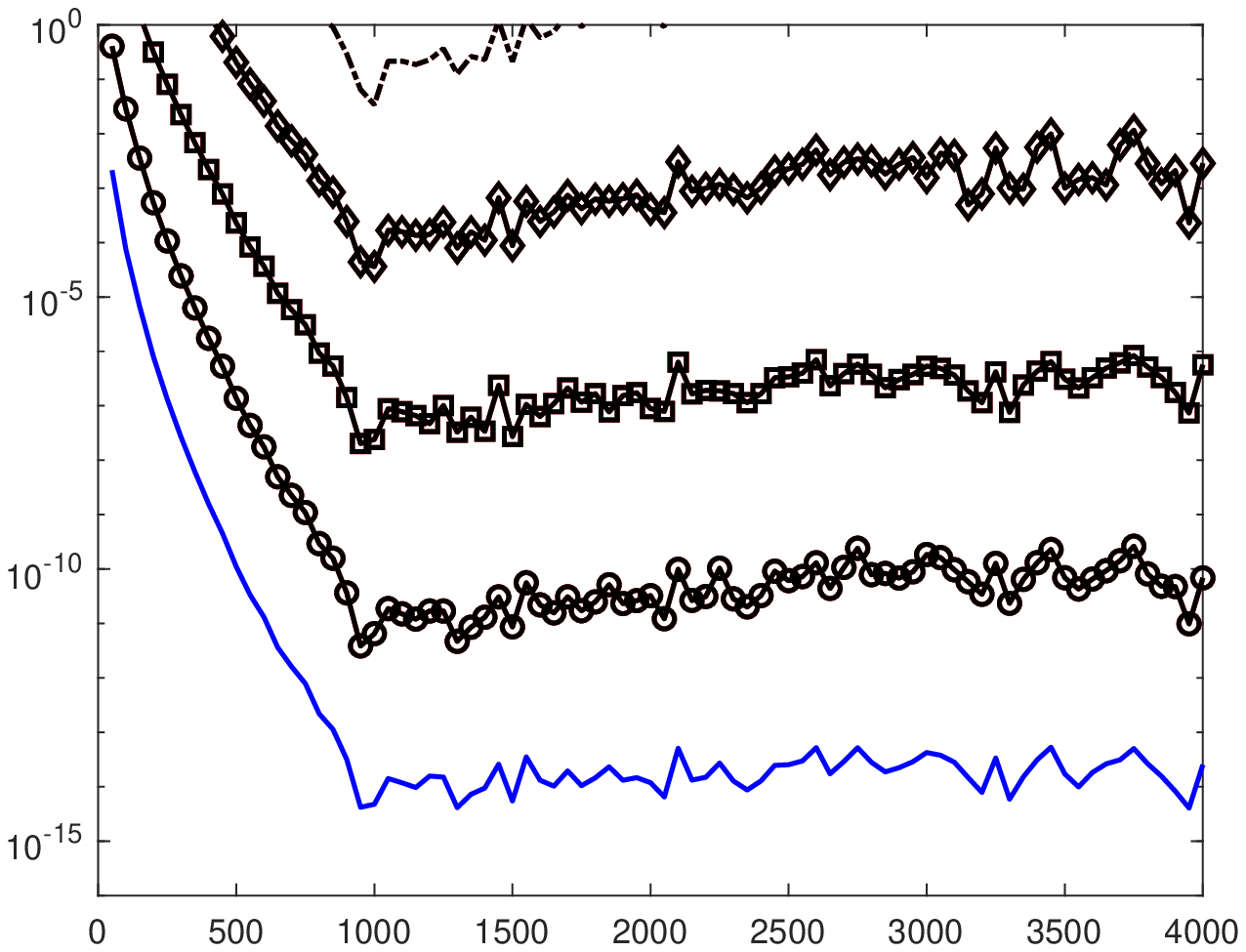}
         \caption{Mean approximation error (left) and Maximum approximation error (right) at $N=100000$ random points in the interval $[-1,1]$ relative to the function $f_3(x)=\frac{1}{1+8x^2}$ (in blue) and its first four derivatives, with the increasing order from the bottom to the top, when approximated by using the constrained mock-Chebyshev least squares operator with $n=50k$, $k=1,\dots,80$.  From the plots, it is evident that the application of the two strategies $S_1$ (red) and $S_2$ (black) gives practically the same results.}
    \label{figureErrorf1}
    \end{center}
\end{figure}

\begin{figure}      
\begin{center}
\includegraphics[width=.49\textwidth]{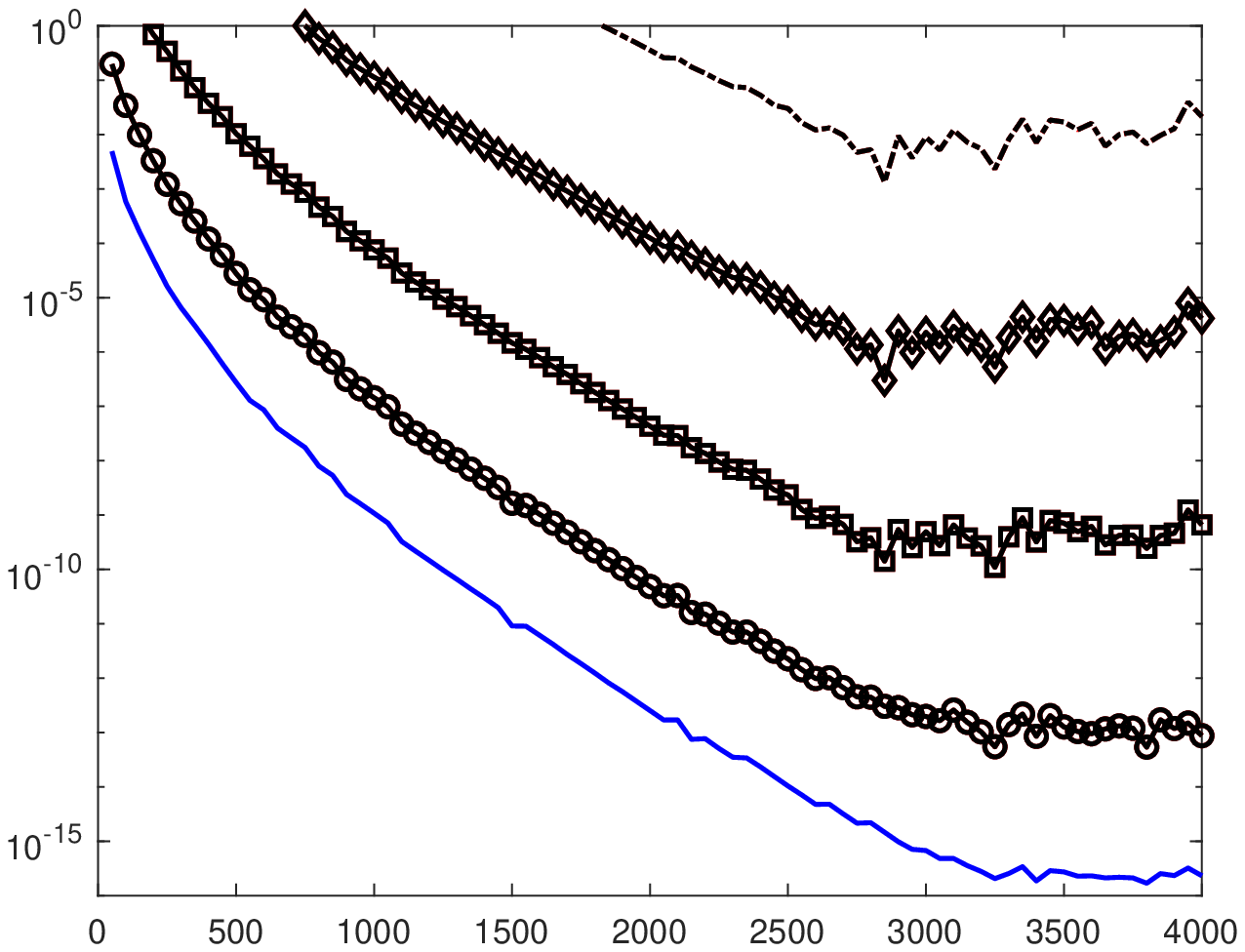}
  \includegraphics[width=.49\textwidth]{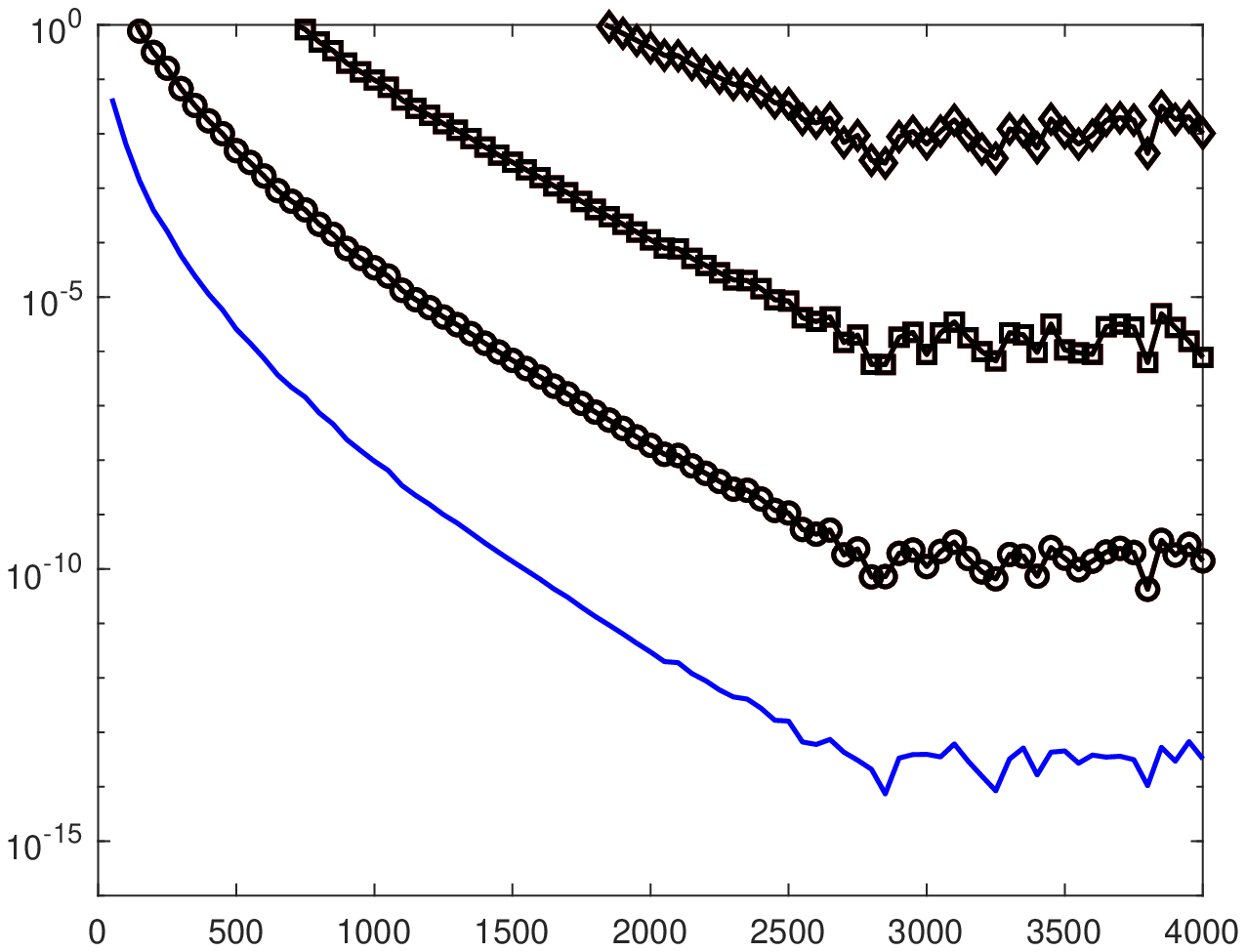}
         \caption{Mean approximation error (left) and Maximum approximation error (right) at $N=100000$ random points in the interval $[-1,1]$ relative to the function  $  f_4(x)=\frac{1}{1+25x^2}$ (in blue) and its first four derivatives, with the increasing order from the bottom to the top, when approximated by using the constrained mock-Chebyshev least squares operator with $n=50k$, $k=1,\dots,80$.  From the plots, it is evident that the application of the two strategies $S_1$ (red) and $S_2$ (black) gives practically the same results.}
    \label{figureErrorf4}
    \end{center}
\end{figure}

\section{Conclusions }
In this paper, we have analyzed the theoretical aspects of the constrained mock-Chebyshev least squares operator. We have introduced explicit representations of the error and its derivatives. By using the constrained mock-Chebyshev least squares operator, we have presented a method for approximating the successive derivatives of $f$ at any point $x\in [-1,1]$ and provided 
estimates for these approximations. This formula provides a global polynomial approximation of the successive derivatives of the function $f$.  

\section*{Acknowledgments}
 This research has been achieved as part of RITA \textquotedblleft Research
 ITalian network on Approximation'' and as part of the UMI group ``Teoria dell'Approssimazione
 e Applicazioni''. The research was supported by GNCS-INdAM 2022 projects. The authors are members of the INdAM Research group GNCS.

\bibliographystyle{abbrv}
\bibliography{bibliography}

\end{document}